\documentclass[12pt]{article}
\usepackage[utf8]{inputenc}
\usepackage{amsmath,amsthm,amssymb}
\usepackage{mathrsfs}
\usepackage[usenames]{color} 
\DeclareMathOperator{\qdim}{qdim}

\DeclareMathOperator{\tr}{tr}

\DeclareMathOperator{\glob}{glob}
\begin{document}
\input amssym.def
\setcounter{equation}{0}
\newcommand{\wt}{{\rm wt}}
\newcommand{\spa}{\mbox{span}}
\newcommand{\Res}{\mbox{Res}}
\newcommand{\End}{\mbox{End}}
\newcommand{\Ind}{\mbox{Ind}}
\newcommand{\Hom}{\mbox{Hom}}
\newcommand{\Mod}{\mbox{Mod}}
\newcommand{\Rep}{\mbox{Rep}}
\newcommand{\m}{\mbox{mod}\ }
\renewcommand{\theequation}{\thesection.\arabic{equation}}
\numberwithin{equation}{section}

\def \End{{\rm End}}
\def \Aut{{\rm Aut}}
\def \Z{\mathbb Z}
\def \H{\mathbb H}
\def \M{\Bbb M}
\def \C{\mathbb C}
\def \R{\mathbb R}
\def \Q{\mathbb Q}
\def \N{\mathbb N}
\def \ann{{\rm Ann}}
\def \<{\langle}
\def \o{\omega}
\def \O{\Omega}
\def \Or{\cal O}
\def \M{{\cal M}}
\def \1t{\frac{1}{T}}
\def \>{\rangle}
\def \t{\tau }
\def \a{\alpha }
\def \e{\epsilon }
\def \l{\lambda }
\def \L{\Lambda }
\def \g{\gamma}
\def \b{\beta }
\def \om{\omega }
\def \o{\omega }
\def \ot{\otimes}
\def \cg{\chi_g}
\def \ag{\alpha_g}
\def \ah{\alpha_h}
\def \ph{\psi_h}
\def \S{\cal S}
\def \nor{\vartriangleleft}
\def \V{V^{\natural}}
\def\voa{vertex operator algebra\ }
\def \vosa{vertex operator superalgebra\ }
\def \vosas{vertex operator superalgebras\ }
\def \voas{vertex operator algebras}
\def \v{vertex operator algebra\ }
\def \1{{\bf 1}}
\def \be{\begin{equation}\label}
\def \ee{\end{equation}}
\def \qed{\mbox{ $\square$}}
\def \pf {\noindent {\bf Proof:} \,}
\def \bl{\begin{lem}\label}
\def \el{\end{lem}}
\def \ba{\begin{array}}
\def \ea{\end{array}}
\def \bt{\begin{thm}\label}
\def \et{\end{thm}}
\def \br{\begin{rem}\label}
\def \er{\end{rem}}
\def \ed{\end{de}}
\def \bp{\begin{prop}\label}
\def \ep{\end{prop}}
\def \p{\phi}
\def \d{\delta}
\def \irr{{\rm Irr}}
\def \glob{{\rm glob}}
\def \obj{{\rm obj}}
\def \Id{{\rm Id}}
\def\OO{\mathcal{O}}
\def\EE{\mathcal{E}}
\def\Vl0{{(V^l)_{\bar 0}}}
\def\Vm0{{(V^m)_{\bar 0}}}
\def\V0{{V_{\bar 0}}}
\def\VlZ{{V(l,\BZ)}}
\def\VlZpm{{V_{\pm}(l,\BZ) }}
\def\VlhZ{{V(l,\BZ+\frac{1}{2})}}
\newtheorem{th1}{Theorem}
\newtheorem{ree}[th1]{Remark}
\newtheorem{thm}{Theorem}[section]
\newtheorem{prop}[thm]{Proposition}
\newtheorem{coro}[thm]{Corollary}
\newtheorem{lem}[thm]{Lemma}
\newtheorem{rem}[thm]{Remark}
\newtheorem{de}[thm]{Definition}
\newtheorem{hy}[thm]{Hypothesis}
\newtheorem{conj}[thm]{Conjecture}
\newtheorem{ex}[thm]{Example}
\newtheorem{q}[thm]{Question}

\newcommand\red[1]{{\color{red} #1}}
\def\b1{\mathbf{1}}
\def\BZ{\mathbb{Z}}
\def\ol{\overline}
\def\id{\operatorname{id}}
\def\sV{\operatorname{sVec}}
\def\s{\sigma}
\def\ft{\frak t}
\def\uc{\underline{c}}
\def\CC{\mathcal{C}}
\def\SC{\mathcal{SC}}
\def\uSC{\underline{\SC}}
\def\DD{\mathcal{D}}
\def\BB{\mathcal{B}}
\def\ZZ{\mathcal{Z}}
\def\CV0{\mathcal{C}_{V_{\bar 0}}}
\begin{center}
{\Large {\bf  Vertex operator superalgebras and  16-fold way}}\\
\vspace{0.5cm}

Chongying Dong\footnote
{Partially supported by China NSF grant 11871351}
\\
Department of Mathematics, University of
California, Santa Cruz, CA 95064 USA \\

Siu-Hung Ng\footnote{Partially supported by NSF grants DMS1001566, DMS1303253, and
DMS1501179}\\
Department of Mathematics
Louisiana State University
Baton Rouge, LA 70803

Li Ren\footnote{Supported by China NSF grant 11671277}\\
 School of Mathematics,  Sichuan University,
Chengdu 610064 China
\end{center}

\begin{abstract} Let  $V$ be a vertex operator superalgebra with the natural order 2 automorphism $\s$.  Under suitable conditions on $V$, the $\s$-fixed subspace $V_{\bar 0}$ is a vertex operator algebra and the category $\CV0$ of $V_{\bar 0}$-modules is modular tensor category. In this paper, we prove that $\CV0$ is a fermionic modular tensor category and the M\"uger centralizer $\CV0^0$ of the fermion in  $\CV0$ is generated by the irreducible $V_{\bar 0}$-submodules of the $V$-modules.  In particular, $\CV0^0$  is a super-modular tensor category and $\CV0$ is a minimal modular extension of $\CV0^0$. We provide a construction of a \vosa $V^l$ for each positive integer $l$ such that  $\CC_{\Vl0}$ is minimal modular extension of $\CV0^0$. We prove that these modular tensor categories $\CC_{\Vl0}$ are uniquely determined, up to equivalence, by the congruence class of  $l$ modulo 16.
\end{abstract}

\section{Introduction}

Modular (tensor) categories are mathematical formalization of topological phases of matters, which are also called topological orders \cite{W}. The 2+1D symmetry protected topological (SPT) orders are recently described by the using unitary braided fusion categories $\CC$ with the symmetry determined by their  M\"uger center $\EE$, which are symmetric fusion categories (cf. \cite{LKW1, LKW2} and the references therein). It follows from \cite{De, DR} that $\EE$ are Tannakian or super-Tannakian, i.e. $\EE$ is equivalent to the braided fusion category $\Rep(G)$ or $\Rep(G,z)$ where $G$ is a finite group uniquely determined by $\EE$ and $z$ is a central order 2 element of $G$.  Modular tensor categories are exactly those braided fusion categories with trivial M\"uger centers.     The category $\sV$ of super vector spaces over $\C$ is the \emph{smallest} super-Tannakian category. By gauging the minimal topolocial order with the fermionic symmetry \cite{Ki}, Kitaev discovered the 16-fold way: The braided fusion category $\sV$ has 16 exactly inequivalent unitary minimal modular extensions, which are unitary modular tensor categories of dimension 4 containing a full braided fusion subcategory equivalent to $\sV$.

A super-modular category  throughout this paper means a  braided fusion category over $\C$ whose M\"uger center is equivalent to $\sV$. Throughout this paper, modular or super-modular categories are assumed to be \emph{pseudounitary} and equipped with the canonical pivotal structures, i.e. the categorical (or quantum) dimension of each object is a positive number.  Motivated by the Kitaev's 16-fold way, it is conjectured in \cite{BGHNPRW} that every super-modular category $\CC$ has exactly 16 minimal modular extensions up to equivalence, i.e. pseudounitary modular categories of dimension $2\dim(\CC)$ containing a braided fusion full subcategory equivalent to $\CC$.  If $\CC$ admits a minimal modular extension, it has been proved independently in \cite{LKW1} that $\CC$ admits 16 minimal modular extensions. However, the existence of minimal modular extension for any super-modular category is still an open problem.

Rational conformal field theory is closely related to modular categories.  According to \cite{Hu1, Hu2}, the representation category of a rational $C_2$-cofinite vertex operator algebra (VOA) is modular. In fact, it is an open problem whether every modular category over $\C$ can be realized by a VOA. Super-modular categories are not modular, and so they cannot be realized as the module category of any rational VOA. One would ask what kind of rational VOA could realize a minimal modular extension of super-modular category $\CC$, and how one can obtain other VOAs whose module categories are minimal modular extensions of the $\CC$.

A vertex operator superalgebra $V=V_{\bar 0} \oplus V_{\bar 1}$ is a VOA  equipped with a $\BZ_2$-graded structure. The $\BZ_2$-grading determines a natural order 2 automorphism $\s$ of $V$ and the component $V_{\bar 0}$ is the sub VOA of $V$ fixed by $\s$. The twisted representations and orbifold theory of rational vertex operator superalgebras are well-studied in \cite{DZ1, DZ2}.   With suitable assumptions (A1 and A2 in Section \ref{s4}) on the vertex operator superalgebra $V$, the $V_{\bar 0}$-module category, denoted by $\CV0$, is a modular tensor category and $V_{\bar 1}$ is a \emph{fermion} of $V_{\bar 0}$ (cf. Lemma \ref{l9.1}). In particular, $V_{\bar 1}$ is an order 2 simple current of $V_{\bar{0}}$. We prove in Theorem \ref{th8.1} that the full subcategory $\CV0^0$ of $\CV0$, generated by the simple $V_{\bar 0}$-submodules of  $V$-modules, are closed under the tensor product of $\CV0$. In particular, $\CV0^0$ is a braided fusion subcategory of $\CV0$ with the fermion $V_{\bar 1}$. Moreover, $\CV0^0$ is the M\"uger centralizer of the fermion $V_{\bar 1}$ in $\CV0$. Hence,  $\CV0^0$ is super-modular (cf. Lemma \ref{l9.2}), and $\CV0$ is a minimal modular extension $\CV0^0$. The modular category $\CV0$ is also $\BZ_2$-graded with $\CV0= \CV0^0 \oplus \CV0^1$ where $\CV0^1$ is the full subcategory of $\CV0$ generated by the irreducible $V_{\bar 0}$-submodules of the $\s$-twisted $V$-modules, and $\dim(\CV0^1)=\dim(\CV0^0)$  (cf. Section 8).

  Since a \emph{nice} vertex operator superalgebra $V$ naturally yields a super-modular category $\CV0^0$ and a minimal modular extension $\CV0$, one would like to construct other vertex operator superalgebras from $V$ to realize the 16-fold way of the super-modular category $\CV0^0$. To the goal, we establish in  Theorem \ref{t9.6} that if $U$ is a holomorphic vertex operator superalgebra, then  $V \otimes U$ is a  vertex operator superalgebra and $\CV0^0$ equivalent to $\CC^0_{(V \ot U)_{\bar 0}}$ as braided fusion categories. In particular, $\CC_{(V \ot U)_{\bar 0}}$ is another minimal modular extension of $\CV0^0$.

  For each positive integer $l$, there is a nice holomorphic vertex operator superalgebra $V(l, \BZ+\frac{1}{2})$ (cf.  \cite{FFR}, \cite{KW}, \cite{L1}).  For any nice vertex operator superalgebra $V$, then tensor product vertex operator superalgebra $V^l:=V \ot V(l, \BZ+\frac{1}{2})$ provides the super-modular  category $\CC_{\Vl0}^0$ and its minimal modular extension $\CC_{\Vl0}$. Since $\CC_{\Vl0}^0$ is equivalent to $\CV0^0$ as braided fusion category, $\CC_{\Vl0}$ is a minimal modular extension of $\CV0^0$ for each positive integer $l$. We prove in Theorem \ref{t10.3} that $\CC_{\Vl0}$ and $\CC_{\Vm0}$ equivalent modular categories if and only if $m \equiv l \pmod{16}$ by computing their Gauss sums and applying \cite[Theorem 5.4]{LKW2}.

 The paper is organized as follows: An introduction and an overview of the results established in this paper are provided in Section 1. A review of vertex operator superalgebras and some of basic results on their representation theory are presented in Section 2. In Section 3, we introduce the tensor product vertex operator superalgebras and investigate their irreducible representations via the representations their Zhu's superalgebras. We discuss the modular invariance of the trace functions in the orbifold theory for the \vosas in Section 4.  In Section, 5, the irreducible $V_{\bar 0}$-modules of a vertex operator superalgebra $V$ are determined in terms of the irreducible $V_{\bar 0}$-submodules of $V$-modules and twisted $V$-modules.   In Section 6, we show that the associated representation of  $SL_2(\Z)$ on the trace functions in the orbifold theory for the \vosas provided in Section 4 is unitary. Some important relations between the quantum dimensions of the irreducible  $V$-modules and the irreducible $V_{\bar 0}$-modules are established in Section 7.  In Section 8, we prove that the category $\CV0$ of is $\BZ_2$-graded, where $\CV0^0$ and $\CV0^1$ are respectively generated by the irreducible $V_{\bar 0}$-submodules of $V$-modules and $\s$-twisted $V$-modules. We further prove that $\CV0^0$ is a super-modular category and $\CV0$ is a minimal modular extension of $\CV0^0$ in Section 9. In Section 10,  we construct a sequence of \vosas $V^l$ for each positive integer $l$ such that $\CC_{\Vl0}$ is a minimal modular extension of $\CV0^0$ and these modular categories  $\CC_{\Vl0}$ are uniquely determined by the congruence class of $l$ modulo 16.

\section{Preliminaries}

The various notions of twisted modules for a \vosa following \cite{DZ1}, \cite{DZ2}  are reviewed in this section. The concepts such as  rationality, regularity, and $C_2$-cofiniteness from \cite{Z} and \cite{DLM3} are discussed.

A super vector space  is a $\Bbb Z_{2}$-graded vector space
$U=U_{\bar{0}}\oplus U_{\bar{1}}$.  The vectors  in $U_{\bar{0}}$
(resp. $U_{\bar{1}}$) are called even (resp. odd). An element $u$ in $U_{\bar i}$ for some $i=0,1$ will be called $\BZ_2$-homogeneous. In this case, we define $\tilde{u} = \bar{i}$. We reserve the notation $\sV$ for the category of finite dimensional super vector spaces over $\C$ with morphisms preserving the $\BZ_2$-gradings, and equipped with the super braiding.

 If $W$ is another super vector space, then $\Hom(U,W)$ is also a super vector space in which $\Hom(U,W)_{\bar{0}}$ and $\Hom(U,W)_{\bar{1}}$ are respectively the $\Z_2$-graded preserving and reversing linear maps.

 A  \emph{vertex operator superalgebra} is a
$\frac{1}{2}\Bbb Z$-graded super vector space
\begin{equation*}
V=\bigoplus_{n\in{ \frac{1}{2}\Bbb Z}}V_n= V_{\bar{0}}\oplus V_{\bar{1}}
\end{equation*}
with  $V_{\bar{0}}=\sum_{n\in\Z}V_n$ and
$V_{\bar{1}}=\sum_{n\in\frac{1}{2}+\Z}V_n$
satisfying $\dim V_{n}< \infty$ for all $n$ and $V_m=0$ if $m$ is sufficiently
small.  $V$ is   equipped with a linear map
 \begin{align*}
& V \to (\mbox{End}\,V)[[z,z^{-1}]] ,\\
& v\mapsto Y(v,z)=\sum_{n\in{\Z}}v_nz^{-n-1}\ \ \ \  (v_n\in
(\End\, V)_{\tilde v})\nonumber
\end{align*}
and with two distinguished vectors ${\bf 1}\in V_0,$ $\omega\in
V_2$ satisfying the following conditions for $u, v \in V,$ and $m,n\in\Z:$
\begin{align*} \label{0a4}
& u_nv=0\ \ \ \ \ {\rm for}\ \  n\ \ {\rm sufficiently\ large};  \\
& Y({\bf 1},z)=Id_{V};  \\
& Y(v,z){\bf 1}\in V[[z]]\ \ \ {\rm and}\ \ \ \lim_{z\to
0}Y(v,z){\bf 1}=v;\\
& [L(m),L(n)]=(m-n)L(m+n)+\frac{1}{12}(m^3-m)\delta_{m+n,0}c ;\\
& \frac{d}{dz}Y(v,z)=Y(L(-1)v,z);\\
& L(0)|_{V_n}=n
\end{align*}
where $L(m)=\o_{ m+1}, $ that is,
$$Y(\o,z)=\sum_{n\in\Z}L(n)z^{-n-2};$$
and the {\em Jacobi identity} holds:
\begin{equation*}\label{2.8}
\begin{array}{c}
\displaystyle{z^{-1}_0\delta\left(\frac{z_1-z_2}{z_0}\right)
Y(u,z_1)Y(v,z_2)-(-1)^{\tilde{u}\tilde {v}}z^{-1}_0\delta\left(\frac{z_2-z_1}{-z_0}\right)
Y(v,z_2)Y(u,z_1)}\\
\displaystyle{=z_2^{-1}\delta
\left(\frac{z_1-z_0}{z_2}\right)
Y(Y(u,z_0)v,z_2)}
\end{array}
\end{equation*}
where
$\delta(z)=\sum_{n\in {\Bbb Z}}z^n$ and  $(z_i-z_j)^n$ is
expanded as a formal power series in $z_j$, and $u,v \in V$ are $\BZ_2$-homogeneous elements. Throughout the paper,
$z_0,z_1,z_2,$ etc. are independent commuting formal variables. A vertex operator superalgebra will be denoted by  $V=(V,Y,{\bf 1},\omega).$
In the case $V_{\bar 1}=0,$ $V$ is a vertex operator algebra given in [FLM3].

Let $V$ be a vertex operator superalgebra. There is a canonical order 2 linear automorphism $\sigma$ of $V$ associated to the structure of super vector space $V$ such that $\sigma|_{V_{\bar i}}=(-1)^i$ for $i=0,1$. It is easy to show that $\sigma\b1=\b1$, $\sigma \omega=\omega$ and $\sigma Y(v,z)\sigma ^{-1}=Y(\sigma v,z)$ for $v\in V.$ That is, $\sigma$ is an automorphism of vertex operator superalgebra $V.$

Let $g=\sigma^i$ for $i=0,1$ and $T=o(g)$. Let $V^r=\{v\in V|gv=e^{2\pi ir/T}v\}$ for $r=0,T-1.$
 A weak $g$-twisted $V$-module $M$ is a vector space equipped
with a linear map
$$\begin{array}{l}
V\to (\End\,M)[[z^{1/T}, z^{-1/T}]\\
v\mapsto\displaystyle{ Y_M(v,z)=\sum_{n\in\frac{1}{T}\Z}v_nz^{-n-1}\ \ \ (v_n\in
\End\,M)}
\end{array}$$
which satisfies that for all $0\leq r\leq T-1,$ $u\in V^r,$ $v\in V,$
$w\in M,$
\begin{eqnarray*}\label{g2.11}
& &Y_M(u,z)=\sum_{n\in \frac{r}{T}+\Z}u_nz^{-n-1} \label{1/2} ;\\
& &u_lw=0 \ \ \  				
\mbox{for}\ \ \ l>>0\label{vlw0};\\
& &Y_M(\1,z)=Id_{M};\label{vacuum}
\end{eqnarray*}
 \begin{equation*}\label{2.14}
\begin{array}{c}
\displaystyle{z^{-1}_0\delta\left(\frac{z_1-z_2}{z_0}\right)
Y_M(u,z_1)Y_M(v,z_2)-(-1)^{\tilde{u}\tilde{v}}z^{-1}_0\delta\left(\frac{z_2-z_1}{-z_0}\right)
Y_M(v,z_2)Y_M(u,z_1)}\\
\displaystyle{=z_2^{-1}\left(\frac{z_1-z_0}{z_2}\right)^{-r/T}
\delta\left(\frac{z_1-z_0}{z_2}\right)
Y_M(Y(u,z_0)v,z_2)}
\end{array}
\end{equation*}
where we assume that  $u, v$ are $\BZ_2$-homogeneous.

Let $o(g\sigma)=T'.$
An {\em admissible} $g$-twisted $V$-module
is a  weak $g$-twisted $V$-module $M$ which carries a
$\frac{1}{T'}{\Z}_{+}$-grading
\begin{equation*}\label{g2.22}
M=\oplus_{n\in\frac{1}{T'}\Z_+}M(n)
\end{equation*}
satisfying
\begin{eqnarray*}\label{g2.23}
v_{m}M(n)\subseteq M(n+\wt v-m-1)
\end{eqnarray*}
for homogeneous $v\in V.$

 An (ordinary) $g$-{\em twisted $V$-module} is
a weak $g$-twisted $V$-module
\begin{equation*}\label{g2.21}
M=\bigoplus_{\lambda \in{\C}}M_{\lambda}
\end{equation*}
such that $\dim M_{\l}$ is finite and for fixed $\l,$ $M_{n+\l}=0$
for all small enough integers $n$ where
 $M_{\l}=\{w\in M|L(0)w=\l w\}.$
We will write $\wt w=\lambda$ if $w\in M_{\lambda}.$

If $M=\bigoplus_{n\in \frac{1}{T'}\Z_+}M(n)$
is an admissible $g$-twisted $V$-module, the contragredient module $M'$
is defined as follows:
\begin{equation*}
M'=\bigoplus_{n\in \frac{1}{T'} \Z_+}M(n)^{*},
\end{equation*}
where $M(n)^*=\Hom_{\C}(M(n),\C).$ The vertex operator
$Y_{M'}(a,z)$ is defined for $a\in V$ via
\begin{eqnarray*}
\langle Y_{M'}(a,z)f,w\rangle= \langle f,Y_M(e^{zL(1)}(e^{\pi i}z^{-2})^{L(0)}a,z^{-1})w\rangle,
\end{eqnarray*}
where $\langle f,w\rangle=f(w)$ is the natural paring $M'\times M\to \C.$
It follows from \cite{FHL} and \cite{X} that $(M',Y_{M'})$ is an admissible $g$-twisted $V$-module.
We can also define the contragredient module $M'$ for a $g$-twisted $V$-module $M.$ In this case,
$M'$ is also a $g$-twisted $V$-module. Moreover, $M$ is irreducible if and only if $M'$ is irreducible.

A \vosa $V$ is called \emph{$g$-rational}, if the category of its admissible $g$-twisted modules is semisimple.  We simply call $V$ rational if $V$ is $1$-rational. $V$ is called holomorphic if $V$ is rational and $V$ is the only irreducible
module for itself up to isomorphism.

We also need  another important concept called $C_2$-cofiniteness \cite{Z}.
We say that a \vosa $V$ is $C_2$-cofinite if $V/C_2(V)$ is finite dimensional, where $C_2(V)=\langle v_{-2}u|v,u\in V\rangle.$  A \vosa $V$ is called regular if every weak $V$-module is a direct sum of irreducible $V$-modules.

The following results about $\sigma^i$-rational  are given in \cite{DZ1} and \cite{DZ2}. Also see \cite{DLM4} and \cite{DLM7}.
\begin{thm}\label{grational}
Let $V$ be a $g$-rational vertex operator superalgebra where $g=\sigma^i$ and $i =0,1$. Then:

(1) Any irreducible admissible $g$-twisted $V$-module $M$ is an ordinary $g$-twisted $V$-module. Moreover, there exists a number $\l \in \mathbb{C}$ such that  $M=\oplus_{n\in \frac{1}{T'}\mathbb{Z_+}}M_{\l +n}$ where $M_{\lambda}\neq 0.$ The $\l$ is called the conformal weight of $M$.

(2) There are only finitely many irreducible admissible  $g$-twisted $V$-modules up to isomorphism.

(3) If $V$ is also $C_2$-cofinite and $\sigma^i$-rational for $i=0,1$ then the central charge $c$ and the conformal weight $\l$ of any irreducible $\sigma^i$-twisted $V$-module $M$ are rational numbers.
\end{thm}

A \vosa $V=\oplus_{n\in \frac{1}{2}\Z}V_n$  is said to be of \emph{CFT type} if $V_n=0$ for negative
$n$ and $V_0=\C {\bf 1}.$
We know from  \cite{L3} and \cite{ABD} that if  $V$ is a vertex operator algebra  of CFT type, then regularity is equivalent to rationality and $C_2$-cofiniteness.  Moreover, $V$ is regular if and only if the weak module category is semisimple \cite{DYu}. The same results also hold for vertex operator superalgebras with similar  proof \cite{HA}.

We discuss more on $V$-modules.  Let $M=\oplus_{n\in\frac{1}{2}\Z_+}M(n)$ be an admissible $V$-module. We set
$M_{\bar 0}=\oplus_{n\in\Z_+}M(n)$ and $M_{\bar 1}=\oplus_{n\in\Z_+}M(n+\frac{1}{2}).$  From now on we assume that $V$ is a simple \vosa and $V_{\bar 1}\ne 0.$ Then $V_{\bar 0}$ is a simple \voa and $V_{\bar 1}$ is an irreducible $V_{\bar 0}$-module.

\begin{lem}\label{l2.1} Let $M=(M,Y_M)$ be a nonzero  admissible  $V$-module. Then $M_{\bar i}\ne 0$ for $i=0,1.$ Moreover,
we can define $\sigma$ action on $M$ such that $\sigma|_{M_{\bar i}}=(-1)^i$ and $\sigma Y_M(u,z)\sigma^{-1}=Y_M(\sigma u,z)$ for all $u\in V.$
\end{lem}

\pf We can assume $M_{\bar 0}\ne 0.$ Then for any $u\in V_{\bar i}$ and $n\in \Z,$ $u_nM_{\bar 0}\in M_{\bar i}.$ Moreover,
we know if $u\ne 0$ then  $u_nM_{\bar 0}\ne 0$ by Proposition 11.9 of \cite{DL1}. This implies that $M_{\bar i}\ne 0$ for
all $i.$ If $u \in V_{\bar{i}}$, it is easy to see that $u_n \in (\End M)_{\bar{i}}$. Therefore the last statement of the Lemma is clear. \qed

Recall from \cite{DLM7} that $M$ is called $\sigma$-stable if $M\circ \sigma$ and $M$ are isomorphic where
$M\circ\sigma$ is a $V$-module such that $M\circ\sigma=M$ as vector spaces and $Y_{M\circ \sigma}(v,z)=Y_M(\sigma v,z)$ for all $v\in V.$ Lemma \ref{l2.1} asserts that for any admissible $V$-module $M,$ $M\circ \sigma$ and $M$ are isomorphic, or $M$ is $\s$-stable.

We now turn our attention to $\sigma$-twisted $V$-module. In this case, an admissible $\sigma$-twisted module $M$ has gradation $M=\oplus_{n\in\Z_+}M(n).$ So we can not use gradation to divide $M$ into even and odd parts. In this case, we have to use $M\circ \sigma.$

\begin{lem}\label{l2.2} Suppose $M$ is an irreducible admissible $\sigma$-twisted $V$-module. If $M\circ \sigma $ and $M$ are not isomorphic, then $M$ is an irreducible
$V_{\bar 0}$-module. If $M\circ \sigma $ and $M$ are isomorphic, then $M$ is a direct sum of
two inequivalent irreducible $V_{\bar 0}$-modules.  In this case, there exists an involution $\sigma  \in {\rm GL}(M)$ such that
$\sigma Y_M(v,z)\sigma^{-1}=Y_M(\sigma v,z)$ for $v\in V$ and the two irreducible $V_{\bar 0}$-modules are the two different eigenspaces of $\sigma$.
\end{lem}

\pf If $M\circ\sigma$ and $M$ are not isomorphic, it follows from the proof of Theorem 6.1 of \cite{DM} that
$M\circ \sigma$ and $M$ are isomorphic irreducible $V_{\bar 0}$-module. If $M\circ \sigma $ is isomorphic to $M,$ we also denote this isomorphism by
$\sigma$ without confusion. Then  $\sigma: M\to M$ is a linear isomorphism such that $\sigma Y_M(v,z)\sigma^{-1}=Y_M(\sigma v,z)$ for $v$ in $ V$ by Schur's Lemma. We can  choose $\sigma$ such that $\sigma^2=1.$ We denote the eigenspace with eigenvalue $(-1)^i$ by $M_{\bar i}.$ Then $M_{\bar i}$ is irreducible $V_{\bar 0}$-module.
\qed

We now introduce the notion of an admissible $\sigma$-twisted super $V$-module.  An admissible $\sigma$-twisted $V$-module $M$ is called an \emph{admissible  $\sigma$-twisted super $V$-module} if $M$ is $\sigma$-stable.
The ordinary $\sigma$-twisted super $V$-module can be defined similarly.
\begin{lem}\label{supert} If $N$ is an admissible $\sigma$-twisted
$V$-module which is not a $\sigma$-stable, then $N\oplus N\circ\sigma$ is an admissible $\sigma$-twisted super $V$-module. Moreover, $N$ is irreducible
if and only if $N\oplus N\circ \sigma$ is an irreducible  admissible $\sigma$-twisted super $V$-module.
\end{lem}
\begin{proof}  For short, we set $\ol N=N\circ \sigma$ and $M=N\oplus \ol N.$ Since $N = \ol N$ as vector spaces, we can define a linear isomorphism $\sigma: M \to M$ by $\sigma(w, w') = (w', w)$ for any $w, w'\in N$. Obviously, $\sigma^2=\id_M$ and one can verify directly that  $\sigma Y_M(u,z)\sigma=Y_M(\sigma u,z)$ for $u\in V.$ Therefore, $M \circ \sigma \cong M$ and  $M$ is an admissible $\sigma$-twisted super $V$-module with
$$M_{\bar 0}=\{w+\sigma w|w\in M\},\ \ M_{\bar 1}=\{w-\sigma w|w\in M\}.$$
Note that $M_{\bar r}$ and $N$ are isomorphic $V_{\bar 0}$-modules for $r=0,1.$ If $N$ is irreducible, then $M_{\ol{0}}$ and $M_{\ol{1}}$  are irreducible $V_{\ol 0}$-modules by Lemma \ref{l2.2}. Let $X\subset M$ be a nonzero  admissible $\sigma$-twisted super $V$-submodule. Then $X=X_{\bar 0}+X_{\bar 1}.$ Without loss, we can assume that $X_{\bar 0}$ is nonzero. Then $X_{\bar 0}$ is a submodule of the irreducible $V_{\bar 0}$-module $M_{\bar 0}.$ Thus $X_{\bar 0}=M_{\bar 0}.$ Since $V$ is simple, for any nonzero $u\in V_{\bar 1}$ and
any nonzero $w\in M_{\bar 0}$ we know $Y(u,z)w$ is nonzero by Proposition 11.9 of \cite{DL1}. This implies $X_{\bar 1}$ is nonzero and equal to $M_{\bar 1}.$ So $X$ has to be $M$ and hence $M$ is an irreducible super $V$-module.
Conversely, if $M$ is super irreducible, take a nonzero  proper admissible $\sigma$-twisted submodule $Z$ of $N.$ It is easy to see that $Z+\sigma(Z)$ is a nonzero proper  admissible $\sigma$-twisted super module. This is a contradiction. The proof is complete.
\end{proof}

\section{Tensor products}

For the purpose later we need to investigate the tensor product $U\otimes V$ of two vertex operator superalgebras $U$ and $V$ and its twisted modules. The tensor product of vertex operator algebras and their modules were studied in \cite{FHL}. In the super case, the tensor product is more complicated. For example, the tensor product $M\otimes N$ of a $\sigma_U$-twisted $U$-module $M$ and a $\sigma_V$-twisted $V$-module $N$ may not be a $\sigma_{U\otimes V}$-twisted $U\otimes V$-module where $\sigma_U$ is the $\sigma$ on $U.$ We will use $\sigma$ for any vertex operator superalgebra if there is no confusion.  So it is necessary to have a detail discussion.

\begin{lem}\label{tensor1}
Let $U,V$ be vertex operator superalgebras. Then

(1) $U\otimes V$ is also a vertex operator superalgebra with
$$(U\otimes V)_{\bar 0}=U_{\bar 0}\otimes V_{\bar 0}+U_{\bar 1}\otimes V_{\bar 1},\ \ (U\otimes V)_{\bar 1}=U_{\bar 0}\otimes V_{\bar 1}+U_{\bar 1}\otimes V_{\bar 0}$$
and $$Y(u\otimes v,z)(u'\otimes v')=(-1)^{\tilde{v}\tilde{u'}}Y(u,z)u'\otimes Y(v,z)v'$$
for any $\BZ_2$-homogeneous elements $u,u'\in U$ and $v,v'\in V.$

(2) The map $f: U\otimes V\to V\otimes U$ such that $f(u\otimes v)=(-1)^{\tilde{u}\tilde{v}}v\otimes u$ gives an isomorphism of vertex operator superalgebras.

(3) If $M$ is a $\sigma^i$-twisted $U$-module such that $M\circ \sigma^i\cong M$ and $N$ is $\sigma^i$-twisted $V$-module with $i=0,1.$ Then $M\otimes N$ is a $\sigma^i\otimes \sigma^i$-twisted $U\otimes V$-module such that
$$Y(u\otimes v,z)(x\otimes y)=(-1)^{\tilde{v}\tilde{x}}Y(u,z)x\otimes Y(v,z)y$$
$u\in U, $ $v\in V$ and $x\in M$ and $y\in N$ where as usual $\tilde{x}=r$ if $x\in M_{\bar r}.$
In particular, the tensor product $M\otimes N$ of $U$-module $M$ and $V$-module $N$ is a module for $U\otimes V.$

(4) If both $U$ and $V$ are rational, then any irreducible $U\otimes V$-module is isomorphic to $M\otimes N$ for some irreducible $U$-module $M$ and some irreducible $V$-module $N.$

(5) If $M$ is a $\sigma$-twisted super $U$-module and $N$ is a  $\sigma$-twisted super $V$-module then $M\otimes N$ is a  $\sigma\otimes \sigma$-twisted super $U\otimes V$-module with
$$(M\otimes N)_{\bar 0}=M_{\bar 0}\otimes N_{\bar 0}+M_{\bar 1}\otimes N_{\bar 1},\ \ (M\otimes N)_{\bar 1}=M_{\bar 0}\otimes N_{\bar 1}+M_{\bar 1}\otimes N_{\bar 0}.$$
\end{lem}

\begin{proof} The proofs of (1)-(4) is fairly standard \cite{FHL}. (5) follows from (3).
\end{proof}

We deal with the tensor product of $\sigma$-twisted modules next. From Lemma \ref{tensor1} we need to understand $M\otimes N$ where both $M$ and $N$ are not $\sigma$-stable in terms of the tensor product of $A_\sigma(U)$ and $A_{\sigma}(V)$ studied in \cite{DZ2}. For this purpose, we need some basic facts on superalgebras and their super modules from \cite{Kl}.

Let ${\cal A}={\cal A}_{\bar 0}+{\cal A}_{\bar 1}$ be a superalgebra. A super ${\cal A}$-module $M$ is defined as a $\Z_2$-graded module $M=M_{\bar 0}\oplus M_{\bar 1}$ such that ${\cal A}_{\bar r}M_{\bar s}\subset M_{\overline{r+s}}.$ ${\cal A}$ is called  semisimple if ${\cal A}$ is completely reducible super ${\cal A}$-module. ${\cal A}$ is  simple if it is semisimple and the only super ideals are $0$ and itself.

  Here are two types of simple superalgebras ${\cal Q}_k$ (${\cal Q}$ type) and ${\cal M}_{m,n}$ (${\cal M}$ type) for positive integer $k$ and nonnegative integers $m,n$ with $m+n>0.$  The ${\cal Q}_k$ is defined to be a subalgebra of matrix algebra $M_{2k\times 2k}$ consisting of $\left(\begin{array}{cc} A & B\\ -B & A\end{array}\right)$
where $A,B$ are arbitrary $k\times k$ complex matrices,  with $B=0$ for even part and $A=0$ for odd one. The ${\cal M}_{m,n}$ is the full matrix algebra $M_{(m+n) \times (m+n)}.$ Write each matrix as
$\left(\begin{array}{cc} A & C\\ D & B\end{array}\right)$ where $A$ is a $m\times m$ matrix, $B$ is a $n\times n$ matrix, $C$ is a $m\times n$ matrix and $D$ is a $n\times m$ matrix, with $C=0, D=0$ for even part and
$A=0, B=0$ for odd part. Clearly, ${\cal Q}$ type is a direct sum of two copies of a full matrix algebra.

One can find the following results in \cite{Kl}.
\begin{thm}\label{superK} Let ${\cal A}$ be a finite dimensional superalgebra.

(1) The following are equivalent: (a) ${\cal A}$ is a  semisimple superalgebra, (b) ${\cal A}$ is a semisimple associative algebra, (c) ${\cal A}$ is a direct sum of simple superalgebras.

(2) Any finite dimensional simple superalgebra is of either ${\cal Q}$ type or ${\cal M}$ type.

(3) For $k>0,$  ${\cal Q}_k$  has a unique irreducible super module of dimension $2k$ which is a direct sum of two inequivalent ${\cal Q}_k$-modules of dimension $k.$

(4) For $m,n\geq 0$ with $m+n>0,$  ${\cal M}_{m,n}$ has a unique irreducible super module of dimension $m+n$ which is also irreducible ${\cal M}_{m,n}$-module.
\end{thm}

Now we discuss the tensor products of superalgebras and their super modules. Superalgebras are algebras in  $\sV$, which is a braided tensor category. Therefore, the tensor product of two superalgebras is a superalgebra. More precisely, if  ${\cal A}$ and ${\cal B}$ are superalgebras, then ${\cal A} \otimes {\cal B}$ is a superalgebra with
$$({\cal A} \otimes {\cal B})_{\bar 0}= {\cal A}_{\bar 0} \otimes {\cal B}_{\bar 0}+{\cal A}_{\bar 1} \otimes {\cal B}_{\bar 1}, ({\cal A} \otimes {\cal B})_{\bar 1}= {\cal A}_{\bar 0} \otimes {\cal B}_{\bar 1}+{\cal A}_{\bar 1} \otimes {\cal B}_{\bar 0}$$
and
$$(a\otimes b)(a'\otimes b')=(-1)^{\tilde{b}\tilde{a'}}aa'\otimes bb'$$
for any homogeneous elements $a,a'\in {\cal A}$ and $b,b'\in{\cal B}.$ Note that the map $f:{\cal A}\otimes {\cal B}\to {\cal B}\otimes {\cal A}$ with $f(a\otimes b)=(-1)^{\tilde{a}\tilde{b}}b\otimes a$ for $a\in{\cal A}$ and $b\in{\cal B}$ is  the braiding of $\sV$. By \cite{Kl},
$${\cal Q}_{m}\otimes {\cal Q}_n\cong {\cal M}_{mn,mn},\quad {\cal Q}_k\otimes {\cal M}_{m,n}\cong {\cal Q}_{(m+n)k}, \quad {\cal M}_{m,n}\otimes {\cal M}_{k,l}\cong {\cal M}_{mk+nl,ml+nk}$$
as superalgebras or algebras in $\sV$.

We now return to vertex operator superalgebra $V.$ Recall the associative algebra $A_{\sigma}(V)$ from \cite{DZ2}. Let $O_{\sigma}(V)$ to be the subspace of $V$ spanned by $u\circ_\sigma v$ for $u,v\in V$
where
$$u\circ_{\sigma}v=\Res_zY(u,z)v\frac{(1+z)^{\wt u}}{z^2}.$$
Set
$$u*_\sigma v=
\Res_zY(u,z)v\frac{(1+z)^{{\wt}\,u}}{z}$$
and  $A_\sigma(V)=V/O_{\sigma}(V).$
Note that the definition of $A_{\sigma}(V)$ is the same as the Zhu's algebra for a vertex operator algebra.
\begin{thm}\label{dz1} Let $V$ be a vertex operator superalgebra. Then

(1) $A_{\sigma}(V)$ is an associative algebra with product induced from $*_{\sigma}$ on $V$ and identity $\1+O_{\sigma}(V).$ Moreover, $\omega+O_\sigma(V)$ is an central element.

(1') $A_{\sigma}(V)$ is a superalgebra with
$$A_{\sigma}(V)_{\bar r}=(V_{\bar r}+O_\sigma(V))/O_{\sigma}(V)\cong V_{\bar r}/O_{\sigma}(V)\cap V_{\bar r}.$$

(2) If $M=\oplus_{n\geq 0}M(n)$ is an admissible $\sigma$-twisted $V$-module with $M(0)\ne 0$ then $M(0)$ is an $A_\sigma(V)$-module such that $v+O_\sigma(V)$ acts as $o(v)$ where $o(v)=v_{\wt v-1}.$

(2') If $M=\oplus_{n\geq 0}M(n)$ is an  admissible $\sigma$-twisted super $V$-module with $M(0)\ne 0$ then $M(0)$ is a super $A_\sigma(V)$-module such that $v+O_\sigma(V)$ acts as $o(v).$

(3) The assignment, $M\to M(0)$, defines a bijection between inequivalent irreducible admissible  $\sigma$-twisted $V$-modules and inequivalent irreducible  $A_\sigma(V)$-modules.

(3') The assignment, $M\to M(0)$, defines a bijection between inequivalent irreducible  admissible $\sigma$-twisted super $V$-modules and inequivalent irreducible super $A_\sigma(V)$-modules.

(4) If $V$ is $\sigma$-rational then $A_\sigma(V)$ is a finite dimensional semisimple associative algebra.

(4') If $V$ is $\sigma$-rational then $A_\sigma(V)$ is a finite dimensional  semisimple superalgebra.

\end{thm}
\begin{proof} (1)-(4) are given in \cite{DZ2} and the proofs of (1')-(4') can been proved similarly with obvious modifications.
\end{proof}

Now we assume that $V$ is $\sigma$-rational. Let
$$\{N^0, N^0_\sigma ,\dots ,N^q,N^q_\sigma, N^{q+1},\dots ,N^p\}$$
be a complete set of inequivalent irreducible $\sigma$-twisted $V$-modules where $N^i, N^i_\sigma=N^i\circ \sigma$ are inequivalent for $i=0,\dots ,q$ and $N^j\cong N^j\circ \sigma$ for $j=q+1,\dots ,p.$ Then
$$A_\sigma(V)=\bigoplus_{i=0}^q(\End N^i(0)\oplus\End N^i_\sigma(0))\bigoplus \bigoplus_{j=q+1}^p\End N^j(0).$$
For short we denote the $\End N^i(0)\oplus\End N^i_\sigma(0)$ by $A_{\sigma}(V)^i$ for $i=0,\dots ,q$ and  $\End N^j(0)$ by $A_{\sigma}(V)^j$ for $j=q+1,\dots ,p.$
Then $A_{\sigma}(V)=\oplus_{i=0}^pA_{\sigma}(V)^i.$

\begin{lem}\label{l3.4} Let $V$ be a $\sigma$-rational vertex operator superalgebra. If $i=0,\dots ,q,$ $A_\sigma(V)^i$ is a simple  ${\cal Q}$ type superalgebra with unique irreducible super module $N^i(0)\oplus N^i_\sigma(0)$ which is a direct sum of two inequivalent irreducible $A_\sigma(V)^i$-modules
$N^i(0)$ and $N^i_\sigma (0).$ If $i=q+1,\dots ,p,$  $A_\sigma(V)^i$ is a simple  ${\cal M}$ type superalgebra with unique irreducible super module $N^i(0).$
\end{lem}
\begin{proof} By Theorem \ref{dz1}, $A_\sigma(V)^i$ is  semisimple. Clearly, if $i>q,$ $A_\sigma(V)^i$ is a simple ${\cal M}$ type superalgebra with unique irreducible super module $N^i(0).$
If $i\leq q,$ note that $N^i(0)\oplus N_\sigma^i(0)$ is a super  $A_\sigma(V)^i$-module  with $(N^i(0)\oplus N_\sigma^i(0))_{\ol r}$ spanned by $(w, (-1)^r w)$ for $w\in N^i(0).$ Since both $(N^i(0)\oplus N_\sigma^i(0))_{\ol r}$
for $r=0,1$ are isomorphic  irreducible  $A_\sigma(V)_{\ol 0}^i$-modules, we immediately see that  $N^i(0)\oplus N_\sigma^i(0)$ is an irreducible super $A_\sigma(V)^i$ -module  and  $A_\sigma(V)^i$ is a simple
superalgebra of ${\cal Q}$ type. The proof is complete.
\end{proof}

We now can establish the following results on the tensor product of $\sigma$-twisted modules. Let $U$ be another $\sigma$-rational vertex operator superalgebra and
$$\{W^{i'}, W^{i'}_\sigma, W^{j'}\,|\, i'=0,\dots ,q', j'=q'+1,\dots ,p'\}$$
is a complete set of inequivalent irreducible $\sigma$-twisted $U$-modules.
\begin{thm}\label{twist} Let $U,V$ be as above. Then $U\otimes V$ is $\sigma$-rational. Moreover, we have

(1) For $i'=0,\dots ,q', i=0,\dots ,q,$ $(W^{i'}\oplus W^{i'}_\sigma)\otimes (N^{i}\oplus N^{i}_\sigma)$ is a sum of two isomorphic irreducible $\sigma$-twisted $U\otimes V$-modules which are $\sigma$-stable.

(2)  For $i'=0,\dots ,q', j=q+1,\dots ,p,$ $(W^{i'} \oplus W^{i'}_\sigma)\otimes N^j$ is a sum of two inequivalent irreducible $\sigma$-twisted $U\otimes V$-modules $W^{i'}\otimes N^j$  and $W^{i'}_\sigma\otimes N^j$. In particular,
$(W^{i'}\otimes N^j)\circ \sigma\cong  W^{i'}_\sigma\otimes N^j.$

(3) For $j'=q'+1,\dots ,p', i=0,\dots ,q,$ $W^{j'}\otimes (N^{i}\oplus N^{i}_\sigma)$ is a sum of two inequivalent irreducible $\sigma$-twisted $U\otimes V$-modules  $W^{j'}\otimes N^i$   and $W^{j'}\otimes  N^{i}_\sigma$
such that $(W^{j'}\otimes N^i)\circ \sigma\cong W^{j'}\otimes  N^{i}_\sigma.$

(4) For $j'=q'+1,\dots ,p',j=q+1,\dots ,p,$ $W^{j'}\otimes N^j$ is an irreducible $\sigma$-twisted $U\otimes V$-module which is $\sigma$-stable.

(5) Every irreducible $\sigma$-twisted $U\otimes V$-module is isomorphic to one of the irreducible $\sigma$-twisted modules listed in (1)-(4).
\end{thm}

\begin{proof} The proof of $\sigma$-rationality of $U\otimes V$ is similar to that of Proposition 2.7 of \cite{DMZ}. (2)-(4) can be verified directly by Lemma \ref{tensor1}. For (1), we need  $A_{\sigma}(U\otimes V).$ Using the exact proof of Lemma 2.8 in \cite{DMZ} yields $A_{\sigma}(U\otimes V)\cong A_{\sigma}(U)\otimes A_{\sigma}(V).$ This gives
$$A_{\sigma}(U\otimes V)=\bigoplus_{0\leq i'\leq p', 0\leq i\leq p}A_{\sigma}(U)^{i'}\otimes A_{\sigma}(V)^i.$$
Note that these tensor product superalgebras are superalgebras with the multiplication given in the remark after Theorem \ref{superK}.
Using Lemma \ref{l3.4} and the tensor products of simple superalgebras we can give a different proof of (2)-(4). We now prove (1). In this case, $i'\leq q',$ $i\leq q$ and $A_{\sigma}(U)^{i'}\otimes A_{\sigma }(V)^i$ is isomorphic to the simple superalgebra ${\cal M}_{mn,mn}=M_{2mn\times 2mn},$ where $m=\dim W^{i'}(0)$ and $n=\dim N^i(0).$ So  $A_{\sigma}(U)^{i'}\otimes A_{\sigma}( V)^i$ has a unique irreducible module of dimension $2mn.$  Since $(W^{i'}(0) \oplus W^{i'}_\sigma(0))\otimes (N^{i}(0)\oplus N^{i}_\sigma(0))$ is an  $A_{\sigma}(U)^{i'}\otimes A_{\sigma}(V)^i$-module of dimension $4mn$, it has to be a sum of two isomorphic  irreducible super $A_{\sigma}(U)^{i'}\otimes A_{\sigma}( V)^i$-modules. As a result, $(W^{i'}+ W^{i'}_\sigma)\otimes (N^{i}+ N^{i}_\sigma)$ is a sum of two isomorphic irreducible $\sigma$-twisted $U\otimes V$-modules which are $\sigma$-stable. (5) follows from Theorem \ref{dz1} (3).
 \end{proof}

\section{Modular Invariance} \label{s4}

We discuss the modular invariance property of the trace functions in orbifold theory  for \vosa
from \cite{Hu2},  \cite{DZ1} and \cite{DLM7} . Also see \cite{Z}.  We also correct a mistake on the number of irreducible
$\sigma$-twisted $V$-modules in \cite{DZ1}.

For the purpose of the modular invariance, we recall the  vertex operator superalgebra $(V, Y[~], \1, \tilde{\omega})$ associated to a vertex operator superalgebra $V$ defined in \cite{Z}.  Here $\tilde{\omega}=\omega-c/24$ and
$$Y[v,z]=Y(v,e^z-1)e^{z\cdot \wt v}=\sum_{n\in \Z}v[n]z^{-n-1}$$
for homogeneous $v.$ Write
$$Y[\tilde{\omega},z]=\sum_{n\in \Z}L[n]z^{-n-2}.$$
The weight of  $v\in V$  in $(V, Y[~], \1, \tilde{\omega})$ is denoted  by $\wt [v].$

In the rest of this paper, we assume that $V=\oplus_{n\geq 0}V_n$ is a simple vertex operator superalgebra such that

A1. {\em  $V_{\bar 0}$ is regular vertex operator algebra of CFT type,}

A2. {\em The weight of any irreducible $\sigma^i$-twisted $V$-module is positive except for $V$ itself with $i=0,1.$}

Then $V$ is $\sigma^i$-rational for $i=0,1$ by Theorem 4.1 of \cite{DH} and $C_2$-cofinite \cite{ABD}.  Using the arguments from \cite{M} and \cite{CM}  one can show, in fact,  that  $V$ is regular if and only if  $V_{\bar 0}$ is regular.

Denote by $\mathscr{M}(g)$ the inequivalent irreducible $g$-twisted $V$-modules for $g=1,\sigma.$ and set $\mathscr{M}(g,h)=\{M \in \mathscr{M}(g)|  M\circ h \cong M\}$ for $g,h=1,\sigma.$ Note from Lemma \ref{l2.1} that $\mathscr{M}(1,h)=\mathscr{M}(1)$ for $h=1,\sigma.$ Also,  $\mathscr{M}(\sigma,1)=\mathscr{M}(\sigma).$ Then  $\mathscr{M}(g)$  and $\mathscr{M}(g,h)$  are finite sets.

Let $M\in \mathscr{M}(\sigma g,\sigma h)$.
For $v\in V$,  we denote $v_{\wt v-1}$ by $o(v)$ as usual and  set
\begin{equation*}
Z_M(v, (g,h),\tau)=\tr_{_M}o(v)\sigma h q^{L(0)-c/24}=q^{\lambda-c/24}\sum_{n\in\frac{1}{T}\Z_+}\tr_{_{M_{\l+n}}}o(v) \sigma hq^{n}
\end{equation*}
if either $(g,h)\ne (1,\sigma) $ or $(g,h)=(1,\sigma)$ and $M\circ \sigma\cong M,$
 and
\begin{equation*}
Z_M(v, (g,h),\tau)=\frac{1}{\sqrt{2}}\tr_{_M}o(v+\sigma v)q^{L(0)-c/24}=\frac{1}{\sqrt{2}}q^{\lambda-c/24}\sum_{n\in\frac{1}{T}\Z_+}\tr_{_{M_{\l+n}}}o(v+\sigma v)q^{n}
\end{equation*}
if $(g,h)=(1,\sigma)$ and $M\circ \sigma\not\cong M.$ Note that if  $(g,h)=(1,\sigma)$ and $M\circ \sigma\not\cong M$
then
\begin{eqnarray*}
& & Z_M(v, (1,\sigma),\tau)=\frac{1}{\sqrt{2}}(\tr_{_M}o(v)q^{L(0)-c/24}+\tr_{M\circ\sigma}o(v)q^{L(0)-c/24})\\
& &\ \ \ \ =\sqrt{2}\tr_{_M}o(v)q^{L(0)-c/24}\\
& &\ \ \ \ =Z_{M\circ\sigma}(v,(1,\sigma),\tau).
\end{eqnarray*}

The insertion of $\sqrt{2}$ in the definition of $Z_M(v, (1,\sigma),\tau)$ will ensure that the corresponding $S$-matrix is unitary (see the discussion in Section 6).

From \cite{DZ1} we known that  $Z_M(v,(g,h),\tau)$  are holomorphic function on the upper half plane $\H$ with $q=e^{2\pi i\tau}$  \cite{DZ1}.
The definition of  $Z_M(v, (1,\sigma),\tau)$ given in this paper in the case $M\circ \sigma\not\cong M$  is different from \cite{DZ1} where the isomorphism between $M$ and $M\circ \sigma$ was required. This new definition
ensures that  $Z_M(v, (g,h),\tau)$ is a vector in the conformal block ${\cal C}(1,\sigma)$    \cite{DZ1}.  According to the definition
of the conformal block, $Z_M(v,(1,\sigma),\tau)=0$ if $\sigma(v)=-v.$ Clearly, $\tr_{_M}o(v)q^{L(0)-c/24}$ is not necessarily zero
for such $M.$ But $Z_M(v,(1,\sigma),\tau)$ is zero in our new definition.

Define $Z_M(v,\tau)=\tr_Mo(v)q^{L(0)-c/24}$ for $\sigma^s$-twisted $V$-modules $M$ and $s=0,1.$ Then $Z_{M^i}(v,\tau)=Z_{M^i}(v,(\sigma,\sigma),\tau)$ for all $i,$ $Z_{N^j}(v,\tau)=\frac{1}{\sqrt{2}}Z_{N^j}(v, (1,\sigma),\tau)$ if $j=0,\dots ,q$  and $Z_{N^j}(v,\tau)=Z_{N^j}(v, (1,\sigma),\tau)$ if $j=q+1,\dots ,p.$
We also set $\chi_M(\tau)=\tr_Mq^{L(0)-c/24}$ which is called the character of $M.$

\begin{lem}\label{lvanish} If $M\in \mathscr{M}(\sigma g,\sigma h)$ and $v\in V_{\bar 1}$ then $Z_M(v, (g,h),\tau)=0$ for any $g,h.$
\end{lem}
\begin{proof} If $M\circ \sigma\cong M$ the result was obtained in Lemma 6.3 of \cite{DZ1}. It remains to prove the result
if $M$ is an irreducible  $\sigma$-twisted $V$-module $M$ with $M\circ \sigma \not\cong M$.  However, this follows from the preceding discussion.
 \end{proof}

 Let $W$ be the vector space spanned by $Z_M(v,(g,h),\tau)$ for $g,h\in\{1,\sigma\}$ and $M\in \mathscr{M}(\sigma g,\sigma h)$. Then $Z_M$ can be regraded as a function on $V\times \H.$ Now, we define an action of the modular group $\Gamma=SL_2(\Z)$ on $W$ such that
\begin{equation*}
Z_M|_\gamma(v,(g,h),\tau)=(c\tau+d)^{-{\rm wt}[v]}Z_M(v,(g,h), \gamma\tau),
\end{equation*}
where
\begin{equation*}
\gamma: \tau\mapsto\frac{ a\tau + b}{c\tau+d},\ \ \ \gamma=\left(\begin{array}{cc}a & b\\ c & d\end{array}\right)\in\Gamma.
\end{equation*}

Recall that $G=\{1,\sigma\}$ acts on $\mathscr{ M}(1)$ and $\mathscr{M}(\sigma)$ such that $M\to M\circ \sigma.$ We have already known that each $G$-orbit in $\mathscr{M}(1)$ has exactly one module, and each $G$-orbit in $\mathscr{ M}(\sigma)$ has either one or two $\sigma$-twisted modules. Note that if two modules $M^1$ and $M^2$ are in the same $G$-orbit,
then $Z_{M^1}(v,(g,h),\tau)=Z_{M^2}(v,(g,h),\tau)$ for all $v\in V.$ Let ${\cal O}_{\sigma^i}$ be the collection of orbit representations
in $\mathscr{ M}(\sigma^i).$

The following result is essentially obtained in   \cite{DZ1} with suitable modification:
\begin{thm}\label{minvariance} Let  $V$ be a vertex operator superalgebra satisfying the assumptions A1-A2.

(1)  $\{Z_M(v,(g,h),\tau)|M\in {\cal O}_{\sigma g}\}$ is linearly independent.

 (2) There is a representation $\rho_V: \Gamma\to GL(W)$
such that for $g,h\in\{1,\sigma\}$   $\gamma =\left(\begin{array}{cc}a & b\\ c & d\end{array}\right)\in \Gamma,$
and $M\in {\cal O}_{\sigma g},$
$$
Z_{M}|_{\gamma}(v,(g,h),\tau)=\sum_{N\in {\cal O}_{\sigma g^ah^c}} \gamma_{M,N}^{(g,h)} Z_{N}(v,(g^ah^c, g^bh^d), ~\tau)$$
where $\rho(\gamma)=(\gamma_{M,N}^{(g,h)}).$
That is,
$$Z_{M}(v,(g,h),\gamma\tau)=(c\tau+d)^{{\rm wt}[v]}\sum_{N\in {\cal O}_{\sigma g^ah^c}} \gamma_{M,N}^{(g,h)} Z_{N}(v,(g^ah^c, g^bh^d), ~\tau).$$

(3) The  number of $G$-orbits in  $\mathscr{M}(\sigma)$ or the number of inequivalent irreducible  $\sigma$-twisted super modules is equal to the number of inequivalent irreducible $V$-modules.
\end{thm}

Theorem \ref{minvariance} (3) gives a correction of Theorem 8.6 (2) in \cite{DZ1}.  Let ${\cal C}(g,h)$ be the
vector spaces spanned by $Z_M(v,(g,h),\tau)$ for $M\in {\cal O}_{\sigma g}.$ Then by Theorem \ref{minvariance} (2) we know that ${\cal C}(1,\sigma)$ and  ${\cal C}(\sigma,1)$ have the same dimension by using the matrix $ \left(\begin{array}{cc}0 & 1\\ -1 & 0\end{array}\right).$ So ${\cal O}_{\sigma^i}$ have the same cardinality for $i=0,1.$ In particular,
the number of  inequivalent irreducible  $\sigma$-twisted modules  is always greater than or equal to the number
of  inequivalent irreducible   modules. Two numbers are equal if and only if every irreducible $\sigma$-twisted $V$-module
is $\sigma$-stable. This result is different from that in \cite{DLM7} when $V$ is a vertex operator algebra and $g$ is an order 2 automorphism. Moreover, if we replace the irreducible $\sigma$-twisted modules
by the irreducible  $\sigma$-twisted super modules, the result is the same as in the case of  vertex operator algebra.

If $V=V_{\bar 0}$ is a vertex operator algebra then $\rho_V$ is a  unitary representation of $\Gamma$ and the kernel of $\rho_V$ is a congruence subgroup of $\Gamma$ \cite{Z,DLN}.

We use the free fermion as an example to illustrate Theorem \ref{minvariance}.  Let $A(\frac{1}{2}+\Z)$ be an associative algebra
generated by $a(m)$ with $m\in \frac{1}{2}+\Z$ subject to the relation $a(m)a(n)+a(n)a(m)=2\delta_{m+n,0}$ and $A(\frac{1}{2}+\Z)^+$ be the subalgebra generated by $a(m)$ with $m > 0$.  Consider $\C$ as an $A(\frac{1}{2}+\Z)^+$-module with the trivial action $a(m)\cdot 1=0$ for $m > 0$. Then $V(\frac{1}{2}+\Z)=A(\frac{1}{2}+\Z) \ot_{A(\frac{1}{2}+\Z)^+} \C$ is the unique irreducible highest weight $A(\frac{1}{2}+\Z)$-module. As vector spaces, $V(\frac{1}{2}+\Z)$ is isomorphic to the free exterior algebra $\bigwedge[a(m)\,|\, m\le 0]$. It is well known that
$V(\frac{1}{2}+\Z)$ is rational, $C_2$-cofinite vertex operator superalgebra with only one irreducible module $V(\frac{1}{2}+\Z)$ up to isomorphism
\cite{KW} and \cite{L1}. Moreover, $V(\frac{1}{2}+\Z)$ is generated by $a(-1/2)$ such that $Y(a(-1/2),z)=\sum_{n\in \Z}a(n+1/2)z^{-n-1}.$

The vertex operator superalgebra $V(\frac{1}{2}+\Z)$ has two inequivalent irreducible $\sigma$-twisted modules. To construct these two $\sigma$-twisted modules
we need another  associative algebra $A(\Z)$ generated by $a(m)$ with $m\in\Z$ satisfying the relation $a(m)a(n)+a(n)a(m)=2\delta_{m+n,0}.$ Let $A(\Z)^+$ be the  subalgebra of $A(\Z)$  generated by $a(m)$ with $m>0.$
Consider the induced $A(\Z)$-module $V(\Z)=A(\Z)\otimes_{A(\Z)^+}\C$ where $\C$ is $A(\Z)^+$-module such that $a(m)1=0$ for
all $m>0.$  It is easy to see that $V(\Z)$ is isomorphic to $ \bigwedge [a(n)|n\in \Z, n\le 0]$, in which $a(m)$ acts by multiplication if $m \leq 0$ and $a(m)$ acts as $\pm 2\frac{\partial}{\partial a(-m)}$ if $m>0$.  Let $W=\bigwedge[a(m)\,|\, m \in \Z, m < 0]$ and $W=W_{\bar 0}\oplus W_{\bar 1}$
be the decomposition of $W$ into the sum even and odd subspaces.  Then
$$V(\Z)_{\pm}=(1\pm a(0))W_{\bar 0}\oplus (1\mp a(0))W_{\bar 1}.$$
are irreducible $A(\Z)$-submodules of $V(\Z)$ and
$V(\Z)=V(\Z)_{+}\oplus V(\Z)_{-}$. Moreover, $V(\Z)_{\pm}$ are the inequivalent
irreducible $\sigma$-twisted $V(\frac{1}{2}+\Z)$-module such that $Y(a(-1/2),z)=\sum_{n\in\Z}a(n)z^{-n-1/2}$ \cite{L2}, \cite{DZ2}.
It is easy to verify that $V(\Z)_+\circ\sigma$ is isomorphic to $V(\Z)_-.$ Furthermore, $V(\Z)$ is the unique irreducible $\sigma$-twisted super $V(\frac{1}{2}+\Z)$-module,

Next we want to discuss more on the trace functions $Z_M(v,(g,h),\tau).$ We know from the Lemma \ref{lvanish} that
$Z_M(v,(g,h),\tau)=0$ if $\sigma v=-v.$ But we can still  consider $\tr_Mo(v)q^{L(0)-c/24}$ for $M\in  \mathscr{M}(\sigma)$ such that $M$ and $M\circ \sigma$ are not isomorphic, and $v\in V_{\bar 1}.$ In general,  $\tr_Mo(v)q^{L(0)-c/24}$ does not vanish.
But our result does not tell any thing about such   $\tr_Mo(v)q^{L(0)-c/24}.$ Now consider the example $V(\frac{1}{2}+\Z).$ Let $v=a(-1/2)\in V(\frac{1}{2}+\Z)_{\bar 1}.$ Then $\wt[v]=\frac{1}{2}$ and $o(v)=a(0)$ on the twisted module. It is easy to compute that
 $$\tr_{V(\Z)_{\pm}}o(v)q^{L(0)-c/24}=\pm q^{1/24}\prod_{n=1}^{\infty}(1-q^n)$$
which is a modular form of weight $\frac{1}{2}$ over $\Gamma.$ This suggests that for an arbitrary rational vertex operator superalgebra $V,$ an irreducible $\sigma$-twisted module $M$ and $v\in V_{\bar 1},$ $\tr_Mo(v)q^{L(0)-c/24}$ is still a
modular form of weight $\wt[v].$

The following corollary is immediate.
\begin{coro}\label{smatrix}  If  $\gamma=S=\left(\begin{array}{cc}0 & -1\\ 1 & 0\end{array}\right)$ and $v\in V_{\bar 0}$ we
 have:
\begin{equation*}\label{S-tran1}
Z_{M}(v,(1,1),-\frac{1}{\tau})=\tau^{\wt[v]}\sum_{N\in {\cal O}_{\sigma} }S_{M,N}^{(1,1)} Z_{N}(v,(1,1),\tau),
\end{equation*}
\begin{equation*}\label{S-tran1.5}
Z_{M}(v,(1,\sigma),-\frac{1}{\tau})=\tau^{\wt[v]}\sum_{N\in \mathscr{M}(1) }S_{M,N}^{(1,\sigma)} Z_{N}(v,(\sigma,1),\tau)
\end{equation*}
for $M\in {\cal O}_{\sigma},$ and
\begin{equation*}\label{S-tran2}
Z_{N}(v,(\sigma,1),-\frac{1}{\tau})=\tau^{\wt[v]}\sum_{M\in {\cal O}_{\sigma}} S_{N,M}^{(\sigma,1)} Z_{M}(v,(1,\sigma),\tau),
\end{equation*}
\begin{equation*}\label{S-tran3}
Z_{N}(v,(\sigma,\sigma),-\frac{1}{\tau})=\tau^{\wt[v]}\sum_{M\in \mathscr{M}(1) } S_{N,M}^{(\sigma,\sigma)} Z_{M}(v,(\sigma,\sigma),\tau)
\end{equation*}
for any $N\in  \mathscr{M}(1).$
The matrix $\rho(S)=(S_{M,N}^{(g,h)})$ is called $S$-matrix of $V$ and is independent of the choice of vector $v\in V_{\bar 0}.$
\end{coro}

\begin{rem} If $V_{\bar 1}=0$ then $V=V_{\bar 0}$ is a vertex operator algebra and $\sigma=1.$ In this case, the representation $\rho$ is unitary and the kernel of $\rho$ is a congruence subgroup \cite{DLN}.
\end{rem}

\section{Irreducible $V_{\bar 0}$-modules}

We classify the irreducible $V_{\bar 0}$-modules in this section  and show that any irreducible $V_{\bar 0}$ module occurs in an irreducible $V$-module or an irreducible $\sigma$-twisted module.

Let $\{M^0,\dots , M^p\}$ be inequivalent irreducible $V$-modules with $M^0=V$ and
$$\{N^0,N^0_\s,\dots, N^q,N^q_\s, N^{q+1},\dots,N^{p}\}$$
 be the inequivalent irreducible $\sigma$-twisted $V$-modules such that $N^i$ and $N^i\circ \s$ are equivalent for $i>q.$ Then
$M^i=M^i_{\bar 0}\oplus M^i_{\bar 1}$ and $N^j=N^j_{\bar 0}\oplus N^j_{\bar 1}$ are direct sum of two irreducible $V_{\bar 0}$-modules by Lemmas \ref{l2.1}, \ref{l2.2} for $i=0,\dots ,p$ and $j=q+1,\dots ,p.$

\begin{thm}\label{t4.1}  Let $V$ be a vertex operator superalgebra satisfying the assumptions A1-A2. Then
$$\{M^i_{\bar s},  N^j, N^k_{\bar s}\,|\,i=0,\dots ,p, j=0,\dots, q,k=q+1,\dots ,p, s=0,1\}$$
are inequivalent irreducible $V_{\bar 0}$-modules.
\end{thm}
\begin{proof} We first prove that  $\{M^i_{\bar s}\,|\,i=0,\dots ,p,s=0,1\}$ are inequivalent $V_{\bar 0}$-modules. Following \cite{DLM5}
we can define associative algebras $A_n(V)$ for $n\in\frac{1}{2}\Z_+$ such that $A_0(V)=A(V)$ as defined in \cite{KW} and
both $A_n(V)$ and $A_{n+\frac{1}{2}}(V)$ are the quotient algebras of $A_n(V_{\bar 0})$ for any nonnegative integer
$n.$ Moreover,
$$A_n(V)=\bigoplus_{i=0}^p\bigoplus_{m\leq n}\End M^i(m)$$
as $V$ is rational.
Noting that $M^i_{\bar s}=\oplus_{n\in\Z_+}M^i_{\frac{1}{2}s+n},$ we see immediately that $M^i_{\bar s}$ are inequivalent $V_{\bar 0}$-modules.

 We prove next that $\{ N^j, N^k_{\bar s}\,|\, j=0,\dots, q, k=q+1,\dots ,p, s=0,1\}$ are inequivalent $V_{\bar 0}$-modules.
In this case we need to construct associative algebras $A_{\sigma,n}(V)$  for $n\in\Z_{+}$ following \cite{DLM6} so that $A_{\sigma, 0}(V)=A_{\sigma}(V)$ as defined in \cite{DZ2}.  We can then follow the proof given in \cite{DY} to show that $\{ N^j, N^k_{\bar s}\,|\, j=0,\dots q,k=q+1,\dots ,p, s=0,1\}$ are inequivalent $V_{\bar 0}$-modules.

Finally we prove that any $M^i_{\bar s}$ and $N^j$  or $M^i_{\bar s}$ and $N^k_{\bar t}$  are
not isomorphic.
From Proposition \ref{MTH}, we see that $\qdim_{V_{\bar 0}}V_{\bar 1}=\qdim_VV=1.$  Thus $V_{\bar 1}$ is a simple current \cite{DJX}. This forces $V_{\bar 1}\boxtimes M^i_{\bar s}=M^i_{\overline{ s+1}}$
and  $V_{\bar 1}\boxtimes N^k_{\bar t}=N^k_{\overline{t+1}}$ and   $V_{\bar 1}\boxtimes N^j=N^j$ as $V_{\bar 0}$-modules. Note that
the weight difference between  $M^i_{\bar 0}$ and $M^i_{\bar 1}$ is half integer, and  the weight differences between $N^k_{\bar 0}$ and $N^k_{\bar 1}$ is integer. So any $M^i_{\bar s}$ and $N^k_{\bar t}$ or $M^i_{\bar s}$ and $N^j$ for $i=0,\dots ,p,$ $j=0,\dots ,q,$ $k=q+1,\dots ,p$ and $s,t=0,1$ are
not isomorphic.
\end{proof}

Our next goal is to prove that the irreducible modules given in Theorem \ref{t4.1} is complete.
\begin{thm}\label{t4.2}  Let $V$ be a vertex operator superalgebra satisfying the assumptions A1-A2. Then
$$\{M^i_{\bar s},N^j,  N^k_{\bar s}\, |\,i=0,\dots ,p, j=0,\dots, q,k=q+1,\dots ,p, s=0,1\}$$
is a complete list of  inequivalent irreducible $V_{\bar 0}$-modules.
\end{thm}

\begin{proof}  The main idea in the proof is to use the $S$-matrix for vertex operator algebra $V_{\bar 0}.$ Observe that for $v\in V_{\bar 0},$
$$Z_{V_{\bar 0}}(v,\tau)=\frac{1}{2}(Z_V(v,(\sigma,\sigma),\tau)+Z_V(v,(\sigma,1),\tau)).$$
Thus
$$Z_{V_{\bar 0}}(v,-\frac{1}{\tau})=\frac{1}{2}(Z_V(v,(\sigma,\sigma),-\frac{1}{\tau})+Z_V(v,(\sigma,1),-\frac{1}{\tau}).$$
Using (\ref{S-tran3}) and Theorem \ref{t4.1} we know that
\begin{eqnarray*}
& &Z_V(v,(\sigma,\sigma),-\frac{1}{\tau})=\tau^{\wt [v]}\sum_{i=0}^pS_{V,M^i}^{(\sigma,\sigma)}Z_{M^i}(v,(\sigma,\sigma),\tau)\\
& &\ \ \ \ \ \ =\tau^{\wt [v]}\sum_{i=0}^pS_{V,M^i}^{(\sigma,\sigma)}(Z_{M^i_{\bar 0}}(v,\tau)+Z_{M^i_{\bar 1}}(v,\tau)).
\end{eqnarray*}
By (\ref{S-tran2})  and Theorem \ref{t4.1}
\begin{eqnarray*}
& &Z_V(v,(\sigma,1),-\frac{1}{\tau})=\tau^{\wt [v]}\sum_{M\in {\cal O}_{\sigma}}S_{V,M}^{(\sigma,1)}Z_M(v,(1,\sigma),\tau)\\
& &\ \ \ \ \ \ =\sqrt{2}\tau^{\wt [v]}\sum_{j=0}^qS_{V,N^j}^{(\sigma,1)}Z_{N^j}(v,\tau)+\tau^{\wt [v]}\sum_{j=q+1}^pS_{V,N^j}^{(\sigma,1)}(Z_{N^j_{\bar 0}}(v,\tau)+Z_{N^j_{\bar 1}}(v,\tau)).
\end{eqnarray*}

From \cite{Z},  $Z_{M^i_{\bar r}}(v,\tau),$ $Z_{N^j}(v,\tau),$ $Z_{N^k_{\bar s}}(v,\tau)$ for $i=0,\dots ,p,$ $j=0,\dots ,q$, $k=q+1,\dots , p,$
$r,s=0,1$ are linearly independent vectors in the conformal block of $V_{\bar 0}.$
From \cite{Hu2}, $\tau^{-\wt[v]}Z_{V_{\bar 0}}(v,-\frac{1}{\tau})$ is a linear combination of $Z_W(v,\tau)$ for the irreducible $V_{\bar 0}$-modules $W$
and the coefficient of each  $Z_W(v,\tau)$ in the linear combination is nonzero.  This implies that the list of irreducible $V_{\bar 0}$-modules
in Theorem \ref{t4.1} is complete.
 \end{proof}

\section{The unitarity of $\rho$}

In this section we show that the representation $\rho$ given in Section 4 is unitary. Since the modular group is generated by $S$ and $T=\left(\begin{array}{cc}1&1\\ 0&1\end{array}\right)$, it is good enough to show $\rho_V(S)$ and $\rho_V(T)$ are unitary matrices. Recall that $\rho_{V_{\bar 0}}(S)$ and $\rho_{V_{\bar 0}}(T)$ are the $S$ and $T$ matrices of $V_{\bar 0}.$  The main idea is to use the unitarity of $\rho_{V_{\bar 0}}$ to establish the unitarity of $\rho_V.$ For this purpose we need to
determine the relation between $\rho_V(S)$ and $\rho_{V_{\bar 0}}(S),$ and $\rho_V(T)$ and $\rho_{V_{\bar 0}}(T).$

Recall  that $Z_M(v,\tau)=\tr_Mo(v)q^{L(0)-c/24}$ for an irreducible  $V_{\bar 0}$-module $M$ and $v\in V_{\bar 0}.$ The $S$ and $T$ matrices  of $V_{\bar 0}$ are given defined by
$$Z_M(v,-\frac{1}{\tau})=\tau^{\wt[v]}\sum_{N}S_{M,N}Z_N(v,\tau).$$
$$Z_M(v,\tau+1)=e^{2\pi i(-c/24+\lambda_M)}Z_M(v,\tau)$$
where $N$ runs through the inequivalent irreducible $V_{\bar 0}$-modules, $c$ is the central charge of $V,$ $\lambda_M$ is the lowest weight of $M$.  In particular, the $T$ matrix of $V_{\bar 0}$ is  diagonal with $T_{M,M}=e^{2\pi i(-c/24+\lambda_M)}$
which is a root of unity as both $c$ and $\lambda_M$ are rational \cite{DLM7}.

According to Theorem \ref{t4.2} we have three cases i) $M=M^i_{\bar s}$ for $i=0,\dots ,p$ and $s=0,1$, ii) $M=N^j$ for $j=0,\dots ,q,$ iii) $M=N^k_{\bar s}$ for $k=q+1,\dots ,p$ and $s=0,1.$
We first compute $S_{M^i_{\bar s},N}$ for $i=0,\dots ,p$ and $s=0,1.$ The computation is similar to those given in the proof of Theorem \ref{t4.2} for $v\in V_{\bar 0}:$
\begin{eqnarray*}
&&Z_{M_{\bar s}^i}(v,-\frac{1}{\tau})=\frac{1}{2}(Z_{M^i}(v,(\sigma,\sigma),-\frac{1}{\tau})+(-1)^sZ_{M^i}(v,(\sigma,1),-\frac{1}{\tau}))\\
&&\ \ \ \ =\frac{1}{2}\tau^{\wt[v]}\sum_{j=0}^pS_{M^i,M^j}^{(\sigma,\sigma)}Z_{M^j}(v,(\sigma,\sigma),\tau)+\frac{(-1)^s}{2}\tau^{\wt [v]}\sum_{N\in {\cal O}_{\sigma}}S_{M^i,N}^{(\sigma,1)}Z_N(v,(1,\sigma),\tau)\\
&&\ \ \ \ =\frac{1}{2}\tau^{\wt[v]}\sum_{j=0}^pS_{M^i,M^j}^{(\sigma,\sigma)}(Z_{M^j_{\bar 0}}(v,\tau)+Z_{M^j_{\bar 1}}(v,\tau))\\
&&\ \ \ \ \  \ +\frac{(-1)^s}{\sqrt{2}}\tau^{\wt [v]}\sum_{j=0}^qS_{M^i,N^j}^{(\sigma,1)}Z_{N^j}(v,\tau)+\frac{(-1)^s}{2}\tau^{\wt [v]}\sum_{j=q+1}^pS_{M^i,N^j}^{(\sigma,1)}(Z_{N^j_{\bar 0}}(v,\tau)+Z_{N^j_{\bar 1}}(v,\tau)).
\end{eqnarray*}
The following lemma is immediate.
\begin{lem}\label{l5.1} For $i=0,\dots ,p$ and $s=0,1$ we have

(1) $S_{M^i_{\bar s},M^j_{\bar t}}=\frac{1}{2}S_{M^i,M^j}^{(\sigma,\sigma)}$ for $j=0,\dots ,p$ and $t=0,1,$

(2) $S_{M^i_{\bar s},N^j}=\frac{(-1)^s}{\sqrt{2}}S_{M^i,N^j}^{(\sigma,1)}$ for $j=0,\dots ,q,$

(3)$S_{M^i_{\bar s},N^j_{\bar t}}=\frac{(-1)^s}{2}S_{M^i,N^j}^{(\sigma,1)}$ for $j=q+1,\dots ,p$ and $t=0,1.$
\end{lem}

Next we compute $S_{N^i,N}.$
Since $N^i$  is an irreducible $V_{\bar 0}$-module for $i=0,\dots ,q$ we immediately from Corollary \ref{smatrix} have
\begin{eqnarray*}
&&Z_{N^i}(v,-\frac{1}{\tau})=\frac{1}{\sqrt{2}}Z_{N^i}(v,(1, \sigma),-\frac{1}{\tau})\\
& &\ \ \ \ =\frac{1}{\sqrt{2}}\tau^{\wt[v]}\sum_{j=0}^pS_{N^i,M^j}^{(1,\sigma)}Z_{M^j}(v,(\sigma,1),\tau)\\
& &\ \ \ \ =\frac{1}{\sqrt{2}}\tau^{\wt[v]}\sum_{j=0}^pS_{N^i,M^j}^{(1,\sigma)}(Z_{M^j_{\bar 0}}v,\tau)-Z_{M^j_{\bar 1}}v,\tau)).
\end{eqnarray*}
The discussion above yields
\begin{lem}\label{l5.2} For $i=0,\dots ,q$, $S_{N^i,M^j_{\bar s}}=\frac{(-1)^s}{\sqrt{2}}S_{N^i,M^j}^{(1,\sigma)}$ and $S_{N^i,W}=0$ for the other irreducible $V_{\bar 0}$-modules $W.$
\end{lem}

Similarly, we have
\begin{lem}\label{l5.3} For $i=q+1,\dots ,p$ and $s,t=0,1, $ $S_{N^i_{\bar s},M^j_{\bar t}}=\frac{(-1)^t}{2}S_{N^i,M^j}^{(1,\sigma)}$ for $j=0,\dots ,p$ and
$S_{N^i_{\bar s},N^j_{\bar t}}=\frac{(-1)^{s+t}}{2}S_{N^i,N^j}^{(1,1)}$ for $j=q+1,\dots ,p.$

\end{lem}
\begin{proof} A straightforward calculation using Corollary \ref{smatrix} gives
\begin{eqnarray*}
&&Z_{N_{\bar s}^i}(v,-\frac{1}{\tau})=\frac{1}{2}(Z_{N^i}(v,(1,\sigma),-\frac{1}{\tau})+(-1)^sZ_{N^i}(v,(1,1),-\frac{1}{\tau}))\\
&&\ \ \ \ =\frac{1}{2}\tau^{\wt[v]}\sum_{j=0}^pS_{N^i,M^j}^{(1,\sigma)}Z_{M^j}(v,(\sigma,1),\tau)+\frac{(-1)^s}{2}\tau^{\wt [v]}\sum_{j=q+1}^pS_{N^i,N^j}^{(1,1)}Z_N(v,(1,1),\tau)\\
&&\ \ \ \ =\frac{1}{2}\tau^{\wt[v]}\sum_{j=0}^pS_{N^i,M^j}^{(1,\sigma)}(Z_{M^j_{\bar 0}}(v,\tau)-Z_{M^j_{\bar 1}}(v,\tau))\\
&&\ \ \ \ \  \ +\frac{(-1)^s}{2}\tau^{\wt [v]}\sum_{j=q+1}^pS_{N^i,N^j}^{(1,1)}(Z_{N^j_{\bar 0}}(v,\tau)-Z_{N^j_{\bar 1}}(v,\tau)).
\end{eqnarray*}
The result follows.
\end{proof}

\begin{thm}\label{unitary} The representation $\rho$ given in Theorem \ref{minvariance} is unitary.
\end{thm}

\begin{proof} The unitarity of $\rho(S)$ follows from Lemmas \ref{l5.1}-\ref{l5.3} and the unitarity of $S$ matrix of $V_{\bar 0}.$ It remains to show that $\rho(T)$ is unitary. We have
\begin{eqnarray*}
& &Z_{M^i}(v,(\sigma, \sigma), \tau+1)=Z_{M^i_{\bar 0}}(v,\tau+1)+Z_{M^i_{\bar 1}}(v,\tau+1)\\
& &\ \ =e^{2\pi i(-c/24+\lambda_{M^i_{\bar 0}})}Z_{M^i_{\bar 0}}(v,\tau)-e^{2\pi i(-c/24+\lambda_{M^i_{\bar 0}})}Z_{M^i_{\bar 1}}(v,\tau)\\
& &\ \ =e^{2\pi i(-c/24+\lambda_{M^i_{\bar 0}})}Z_{M^i}(v,(\sigma,1),\tau)
\end{eqnarray*}
where we have used the fact that $\lambda_{M^i_{\bar 0}}-\lambda_{M^i_{\bar 1}}+\frac{1}{2}$ is an integer.
Similarly,
$$Z_{M^i}(v,(\sigma,1), \tau+1)=e^{2\pi i(-c/24+\lambda_{M^i_{\bar 0}})}Z_{M^i}(v,(\sigma,\sigma),\tau).$$
It is easy to see that for $i=0,\dots p$ and $j=q+1,\dots ,p$
$$Z_{N^i}(v,(1,\sigma),\tau+1)=e^{2\pi i(-c/24+\lambda_{N^i})}Z_{N^i}(v,(1,\sigma),\tau),$$
$$Z_{N^j}(v,(1,1),\tau+1)=e^{2\pi i(-c/24+\lambda_{N^j})}Z_{N^j}(v,(1,1),\tau).$$
The unitarity of $\rho(T)$ now follows from that fact that  $c$ and $\lambda_{N^i}$ are rational number \cite{DLM7}.
\end{proof}

\section{Quantum dimensions}
We compute the quantum dimensions of the irreducible $\sigma^i$-twisted $V$-modules and irreducible $V_{\bar 0}$-modules in this section.
The ideas and techniques used here come from \cite{DJX} and \cite{DRX}.

Let $V$ be a vertex operator superalgebra as before and $M$ be an irreducible $\sigma^i$-twisted module.
Recall $\chi_M(\tau)$ from Section 4. The quantum dimension of $M$ over $V$ is defined to be
$$\qdim_{V}M=\lim_{y\to 0}\frac{\chi_M(iy)}{\chi_V(iy)}$$
using the relation $q=e^{2\pi i\tau}$
where $y$ is real and positive.

The existence of the quantum dimension for a $g$-twisted $V$-module is given below in terms of the $S$-matrix and the proof is similar to that of Lemma 4.2 of \cite{DJX} by using the $S$-matrix given in Corollary \ref{smatrix}.
\begin{prop}\label{tqdim}
Let $V$  be a vertex operator superalgebra satisfying A1-A2, and  $M$ an irreducible  $\sigma^r$-twisted $V$-module for $r=0,1.$ Then $\qdim_VM=\frac{S_{M,V}^{(\sigma^{1-r},\sigma)}}{\sqrt{2}S_{V,V}^{(\sigma,\sigma)}}$ if $M=N^i$ for $i=0,\dots ,q$ and
$\qdim_VM=\frac{S_{M,V}^{(\sigma^{1-r},\sigma)}}{S_{V,V}^{(\sigma,\sigma)}}$ for other $M.$ In particular, $\qdim_VM$ exists.
\end{prop}

We define the global dimension
$$\glob(V)=\sum_{M\in \mathscr{M}(1)}(\qdim_VM)^2.$$ In the case $V$ is a vertex operator algebra, this is exactly the global dimension of $V$ defined in \cite{DJX} and is equal to $\frac{1}{S_{V,V}^2}.$

We now  compute the quantum dimensions of irreducible $V_{\bar 0}$-modules in terms of quantum dimensions of irreducible $V$-modules. Recall Theorem \ref{t4.2}.
\begin{prop}\label{MTH}  We have

(1) $\qdim_{V_{\bar 0}}M^i_{\bar r}=\qdim_VM^i$  for $i=0,\dots ,p$ and $r=0,1,$

(2) $\qdim_{V_{\bar 0}}N^j=2\qdim_VN^j$ for $j=0,\dots ,q,$

(3) $\qdim_{V_{\bar 0}}N^k_{\bar s}=\qdim_VN^k$ for $k=q+1,\dots ,p$ and $s=0,1,$

(4)  $\glob(V_{\bar 0})=4\glob(V),$

(5) $\sum_{M\in \mathscr{M}(\sigma)}(\qdim_VM)^2=\glob(V),$

(6) $\sum_{X_1}(\qdim_{ V_{\bar 0}}X_1)^2=\sum_{X_2}(\qdim_{ V_{\bar 0}}X_2)^2$ where $X_i$ ranges over the inequivalent irreducible $V_{\bar 0}$-modules appearing in irreducible $\sigma^i$-twisted $V$-modules,

In particular,  $\qdim_{V_{\bar 0}}W=2\qdim_VW$ for any irreducible $\sigma^r$-twisted module $W.$ Moreover, $ \qdim_VM^i,$ $2\qdim_VN^j$ $j=0,\dots ,q$ and $\qdim_VN^k$ for $k=q+1,\dots ,p$
take values in  $\{2\cos\frac{\pi}{n}|n\geq 3\}\cup [2,\infty).$
\end{prop}
\begin{proof} (1) By Lemma 4.2 of \cite{DJX}, Proposition \ref{tqdim} and Lemma \ref{l5.1} we see that
$$\qdim_{V_{\bar 0}}M^i_{\bar r}=\frac{S_{M^i_{\bar r}, V_{\bar 0}}}{S_{V_{\bar 0}, V_{\bar 0}}}=\frac{S_{M^i,V}^{(\sigma,\sigma)}}{S_{V,V}^{(\sigma,\sigma)}}=\qdim_VM^i.$$

(2) can be proved similarly by using Lemma \ref{l5.2}. But we give a different proof here:
$$\qdim_{V_{\bar 0}}N^j=\lim_{y\to 0}\frac{\chi_{N^j}(iy)}{\chi_{V_{\bar 0}}(iy)}=\lim_{y\to 0}\frac{\chi_{N^j}(iy)}{\chi_V(iy)}\frac{\chi_V(iy)}{\chi_{V_{\bar 0}}(iy)}=\lim_{y\to 0}\frac{\chi_{N^j}(iy)}{\chi_V(iy)}\lim_{y\to 0}\frac{\chi_V(iy)}{\chi_{V_{\bar 0}}(iy)}=2\qdim_VN^j.$$

(3) The proof is similar.

 From \cite{DJX}, the quantum dimensions of irreducible $V_{\bar 0}$-modules lie in  $\{2\cos\frac{\pi}{n}|n\geq 3\}\cup [2,\infty).$ From (1)-(3) we see immediately that
 $ \qdim_VM^i,$ $2\qdim_VN^j$ for $j=0,\dots ,q$ and $\qdim_VN^k$ for $k=q+1,\dots ,p$
take values in  $\{2\cos\frac{\pi}{n}|n\geq 3\}\cup [2,\infty).$

(4) From \cite{DJX} we know that $\glob(V_{\bar 0})=\frac{1}{S_{V_{\bar 0},V_{\bar 0}}^2}$ is positive. This implies that $S_{V_{\bar 0},V_{\bar 0}}$ is a real number. It follows from Proposition \ref{tqdim} that
$S_{V,V}^{(\sigma,\sigma)}$ is a real number. Since $\qdim_VM$ is always positive for any irreducible $\sigma^r$-twisted $V$-module $M$, we see from Proposition \ref{tqdim} again that $S_{M,V}^{(\sigma^{1-r},\sigma)}$ is a real number.

Using  Proposition \ref{tqdim} and Theorem \ref{unitary} yields
$$\glob(V)=\sum_{i=0}^p(\qdim_VM^i)^2=\sum_{i=0}^p(\frac{S_{M^i,V}^{(\sigma,\sigma)}}{S_{V,V}^{(\sigma,\sigma)}})^2=(\frac{1}{S_{V,V}^{(\sigma,\sigma)}})^2\sum_{i=0}^p(S_{M^i,V}^{(\sigma,\sigma)})^2=
(\frac{1}{S_{V,V}^{(\sigma,\sigma)}})^2.$$
By Lemma \ref{l5.1}, $S_{V_{\bar 0}V_{\bar 0}}=\frac{1}{2}S_{V,V}^{(\sigma,\sigma)}.$ It follows immediately that $\glob(V_{\bar 0})=4\glob(V).$

(5)  Again by Proposition \ref{tqdim} and Theorem \ref{unitary}
$$\sum_{M\in \mathscr{M}(\sigma)}(\qdim_VM)^2=\frac{2}{(S_{V,V}^{(\sigma,\sigma)})^2}\sum_{i=0}^q(\frac{S_{N^i,V}^{(1,\sigma)}}{\sqrt{2}})^2+
(\frac{1}{S_{V,V}^{(\sigma,\sigma)}})^2\sum_{i=q+1}^p(S_{N^i,V}^{(1,\sigma)})^2=(\frac{1}{S_{V,V}^{(\sigma,\sigma)}})^2.$$

(6) Note from (1) that
$$\sum_{i=0}^p\sum_{s=0}^1(\qdim_{V_{\bar 0}}M^i_{\bar s})^2=\sum_{i=0}^p2(\qdim_{V}M^i)^2=2\glob(V).$$
 The result follows now from (4) and Theorem \ref{t4.2}.
\end{proof}

\section{$\BZ_2$-grading on the category  of $V_{\bar 0}$-modules}

For the discussion below, we introduce
several module categories. We use ${\cal C}_V$ and ${\cal C}_V^\s$ to denote the $V$-module category and $\sigma$-twisted $V$-module category, respectively. Since $\s|_{V_{\bar 0}}=\id_{V_{\bar 0}}$, the objects in $\CC_V$ and ${\cal C}_V^\s$ are $V_{\bar 0}$-modules. We  denote by ${\cal C}_{V_{\bar 0}}^r$
the full abelian subcategory of ${\cal C}_{V_{\bar 0}}$ generated by the simple $V_{\bar 0}$-submodules of any $\sigma^r$-twisted $V$-modules.  Let ${\cal C}$ be any of these categories, the dimension
of ${\cal C}$ is defined as $\dim {\cal C}=\sum_{M}(\qdim M)^2$ where $M$ runs over the equivalence classes of  simple objects in category ${\cal C}.$ It is clear that $\glob(V)=\dim{\cal C}_V,$  $\glob(V_{\bar 0})=\dim{\cal C}_{V_{\bar 0}}=\dim{\cal C}_{V_{\bar 0}}^0+\dim{\cal C}_{V_{\bar 0}}^1.$ From discussion in Section 7, we know that $\glob(V_{\bar 0})=4\,\glob(V),$ $\dim {\cal C}_V=\dim {\cal C}_V^\s,$ $\dim{\cal C}_{V_{\bar 0}}^0
=\dim{\cal C}_{V_{\bar 0}}^1.$ From \cite{Hu2}, ${\cal C}_{V_{\bar 0}}$ is a modular tensor category.

\begin{thm}\label{th8.1} The category ${\cal C}_{V_{\bar 0}}^0$ is a fusion subcategory of ${\cal C}_{V_{\bar 0}}$ with a complete list of simple objects given by $M_{\bar r}^j$, with $j =0, \dots, p$ and $r=0,1$.
\end{thm}
\begin{proof} We need to show that $M^i_{\bar r}\boxtimes M^j_{\bar s}$ for $i,j=0,\dots ,p$ and $r,s=0,1$ lies in ${\cal C}_{V_{\bar 0}}^0.$ This is equivalent to that the fusion rules $N_{M^i_{\bar r},M^j_{\bar s}}^W=0$
for $W=N^k$ for $k=0,\dots ,q$ or $W=N^k_{\bar t}$ for $k=q+1,\dots ,p$ and $t=0,1.$

From Proposition 4.9 of \cite{DJX} or the Verlinde formula for modular tensor categories,
$$\qdim_{V_{\bar 0}}M^i_{\bar r}\qdim_{V_{\bar 0}}M^j_{\bar s}=\qdim_{V_{\bar 0}}M^i_{\bar r}\boxtimes M^j_{\bar s}=\sum_{W}N_{M^i_{\bar r},M^j_{\bar s}}^{W}\qdim_{V_{\bar 0}}W$$
 where $W$ ranges over the inequivalent irreducible $V_{\bar 0}$-modules. By the assumption of $V$, the quantum dimensions are positive.
Our idea is to establish
$$\qdim_{V_{\bar 0}}M^i_{\bar r}\qdim_{V_{\bar 0}}M^j_{\bar s}=\sum_{k=0}^p\sum_{t=0,1}N_{M^i_{\bar r},M^j_{\bar s}}^{M^k_{\bar t}}\qdim_{V_{\bar 0}}M^k_{\bar t}$$
which implies that $N_{M^i_{\bar r},M^j_{\bar s}}^W=0$ if $W$ is not any $M^k_{\bar t}.$

Recall from \cite{Hu2} the Verlinde formula
$$N_{M^i_{\bar r},M^j_{\bar s}}^{M^k_{\bar t}}=\sum_{W}\frac{S_{M^i_{\bar r},W}S_{M^j_{\bar s},W}\overline{S}_{W,M^k_{\bar t}}}{S_{V_{\bar 0},W}}$$
where $W$ ranges through the inequivalent irreducible $V_{\bar 0}$-modules.
Thus
\begin{eqnarray*}
& &\sum_{k=0}^p\sum_{t=0,1}N_{M^i_{\bar r},M^j_{\bar s}}^{M^k_{\bar t}}\qdim_{V_{\bar 0}}M^k_{\bar t}\\
& &=\sum_{k=0}^p\sum_{t=0,1}\sum_W\frac{S_{M^i_{\bar r},W}S_{M^j_{\bar s},W}\overline{S}_{W,M^k_{\bar t}}}{S_{V_{\bar 0},W}}\frac{S_{M^k_{\bar t},V_{\bar 0}}}{S_{V_{\bar 0},V_{\bar 0}}}
\end{eqnarray*}

We claim that $\sum_{k=0}^p\sum_{t=0,1}\ol S_{W,M^k_{\bar t}}S_{M^k_{\bar t},V_{\bar 0}}$ is  $0$ if $W\ne V_{\bar 0}, V_{\bar 1}$ and  is $\frac{1}{2}$ otherwise.
Note that $\overline{S}_{M^k_{\bar t},W}=S_{W', M^k_{\bar t}}$. Let $W=V_{\bar a}$. By Lemma \ref{l5.1} we have $S_{M^k_{\bar t},V_{\bar a}}=\frac{S_{M^k,V}^{(\sigma,\sigma)}}{2}$ for $a=0,1$. Using the unitarity of $\rho(S)$ Theorem \ref{unitary} gives the claim
$$\sum_{k=0}^p\sum_{t=0,1}\overline{S}_{W,M^k_{\bar t}}S_{M^k_{\bar t},V_{\bar 0}}=\frac{1}{2}\sum_{k=0}^p(S_{M^k,V}^{(\sigma,\sigma)})^2=\frac{1}{2}.$$
The proof for $W=M^k_{\bar t}$ with $k>0$ is similar. For $W=N^i, N^j_{\bar s}$, where $i = 0,\dots, q$, $j=q+1,\dots, p$ and $s=0,1$, the claim follows immediately from Lemmas 6.2 and 6.3.

Finally we have
$$\sum_{k=0}^p\sum_{t=0,1}\sum_W\frac{S_{M^i_{\bar r},W}S_{M^j_{\bar s},W}\overline{S}_{W,M^k_{\bar t}}}{S_{V_{\bar 0},W}}\frac{S_{M^k_{\bar t},V_{\bar 0}}}{S_{V_{\bar 0},V_{\bar 0}}}=\frac{1}{2}\frac{S_{M^i_{\bar r},V_{\bar 0}}S_{M^j_{\bar s},V_{\bar 0}}}{S_{V_{\bar 0},V_{\bar 0}}S_{V_{\bar 0},V_{\bar 0}}}+\frac{1}{2}\frac{S_{M^i_{\bar r},V_{\bar 1}}S_{M^j_{\bar s},V_{\bar 1}}}{S_{V_{\bar 0},V_{\bar 1}}S_{V_{\bar 0},V_{\bar 0}}}.$$
Since $S_{M^i_{\bar r},V_{\bar 1}}=S_{M^i_{\bar r},V_{\bar 0}}$ we see that
$$\sum_{k=0}^p\sum_{t=0,1}N_{M^i_{\bar r},M^j_{\bar s}}^{M^k_{\bar t}}\qdim_{V_{\bar 0}}M^k_{\bar t}=\frac{S_{M^i_{\bar r},V_{\bar 0}}}{S_{V_{\bar 0},V_{\bar 0}}}\frac{S_{M^j_{\bar s},V_{\bar 0}}}{S_{V_{\bar 0},V_{\bar 0}}}
=\qdim_{V_{\bar 0}}M^i_{\bar r}\qdim_{V_{\bar 0}}M^j_{\bar s},$$
as desired. \end{proof}

\begin{rem} Similarly, one can show that  if $M\in {\cal C}_{V_{\bar 0}}^r,$  $N\in {\cal C}_{V_{\bar 0}}^s$ then $M\boxtimes N\in {\cal C}_{V_{\bar 0}}^{r+s}$ where $r+s$ is understood to be modulo $2.$ Therefore, $\CC_{V_{\bar 0}}$ is $\BZ_2$-graded.
\end{rem}

\section{The 16-fold way}

We discuss in this section on how the representation theory for vertex operator superalgebra is related to the 16-fold way conjecture proposed in \cite{BGHNPRW}.

Let $U$ be a rational, $C_2$-cofinite, simple vertex operator algebra of CFT type such that the weight of any irreducible $U$-module is positive except $U$ itself. Then the $U$-module
category ${\cal C}_U$ is a modular tensor category \cite{Hu2} with positive quantum dimensions. As usual, let $c_{M,N}: M\boxtimes N\to N\boxtimes M$ be the braiding for $U$-modules $M,N.$ Let $\theta$ denote ribbon structure on $\CC_U$. Then $\theta_M$ is a scalar multiple of $\id_M$ for any simple $U$-module $M$. We use the abuse notation $\theta_M$ to denote such scalar. A simple $U$-module $F$ is called a \emph{fermion} if $F$ is a simple current (or invertible object of $\CC_U$) of order 2 and $c_{F,F}=-\id_{F \boxtimes F}$. Since $\qdim_U F =1$, $\theta_F=-1$.

\begin{lem}\label{l9.1} Let $V=V_{\bar 0}\oplus V_{\bar 1}$ be a vertex operator superalgebra satisfying assumptions A1-A2 with $V_{\bar 1}\ne 0.$ Then $V_{\bar 1}$ is a fermion of $\CC_{V_{\bar 0}}$.
\end{lem}
\begin{proof} Clearly, $V_{\bar 1}$ is a simple current of order 2. Then
$$c_{V_{\bar 1}, V_{\bar 1}}=\theta_{V_{\bar 1}}\id_{V_{\bar 1}\boxtimes V_{\bar 1}}=e^{2\pi i L(0)}\id_{V_{\bar 1}\boxtimes V_{\bar 1}}=-\id_{V_{\bar 1}\boxtimes V_{\bar 1}}$$
from \cite{Hu2} and  \cite{BGHNPRW} as $V_{\bar 1}=\oplus_{n\in\Z}V_{\frac{1}{2}+n}.$
\end{proof}

Conversely, if $U$ is as before and an $U$-module $F$ is a fermion, then $V=U\oplus F$ has a structure of a vertex operator superalgebra such that $V_{\bar 0}=U$
and $V_{\bar 1}=F$ by Theorem 1.1 of \cite{CKL}. Therefore, the vertex operator superalgebra $V=V_{\bar 0} \oplus V_{\bar 1}$ in our sense is completely determined by a fermion $V_{\bar 1}$ in $\CC_{V_{\bar 0}}$.

Let $\BB$ be a braided fusion category. For any family $\DD$ of objects in $\BB$,  the M\"uger centralizer $C_{\BB}(\DD)$ is the full subcategory of $\BB$ consisting of the objects $Y$ in $\BB$ such that $c_{Y,X}\circ  c_{X,Y} = \id_{X\boxtimes Y}$  for any $X$ in $\DD.$ The subcategory $C_\BB(\DD)$ is closed under the tensor product of $\BB$ and hence a braided fusion subcategory of $\BB$. The symmetric fusion category $C_{\BB}(\BB)$ is call the M\"uger center of $\BB$ and denoted by $\ZZ_2(\BB)$.  In this paper, a pseudounitary braided fusion category $\BB$ is called \emph{super-modular} if  $\ZZ_2(\BB)$ equivalent to category $\sV$, which is equal to $\text{Rep}(\Z_2)$ with the super braiding. In particular, a super-modular category $\BB$ admits a fermion $F$ in $C_\BB(\BB)$ with $\theta_F=-1$.

\begin{lem}\label{l9.2} Let $V=V_{\bar 0}\oplus V_{\bar 1}$ be a vertex operator superalgebra satisfying assumptions A1-A2. Then $\CC_{V_{\bar 0}}^0 = C_{V_{\bar 1}}(\CC_{V_{\bar 0}})$ and is super-modular.
\end{lem}
\begin{proof} By Theorem \ref{th8.1} that $C^0_{V_{\bar 0}}$ is a braided fusion subcategory $\CC_{V_{\bar 0}}$. We first prove that $V_{\bar 1}$ lies in ${\cal Z}_2({\cal C}_{V_{\bar 0}}^0)$, and hence $\dim \ZZ_2(\CC^0_{V_{\bar 0}}) \ge 2$. Equivalently we need to show that
$$c_{V_{\bar 1},M^i_{\bar r}}\circ c_{M^i_{\bar r},V_{\bar 1}}=\id_{M^i_{\bar r}\boxtimes V_{\bar 1}}$$
for $i=0,\dots ,p$ and $r=0,1.$ Since $V_{\bar 1}$ is a simple current we know that both  $V_{\bar 1}\boxtimes M^i_{\bar r}$ and $M^i_{\bar r}\boxtimes V_{\bar 1}$ are isomorphic to $M^i_{\overline{r+1}}.$
As usual we will denote the space of intertwining operator of type $\left(_{W^1,W^2}^{\ \ W^3}\right)$ by $I\left(_{W^1,W^2}^{\ \ W^3}\right)$ where $W^j$ are modules for vertex operator algebra
$V_{\bar 0}$ for $i=1,2,3.$ Then $I\left(_{M^i_{\bar r},V_{\bar 1}}^{\ M^i_{\overline{r+1}}}\right)=\C {\cal Y}$ and $I\left(_{V_{\bar 1},M^i_{\bar r}}^{\ M^i_{\overline{r+1}}}\right)=\C Y$ where $Y$ is the restriction
of $Y$ defining the $V$-module structure on $M^i$ to $V_{\bar 1}$ and ${\cal Y}(w,z)u=e^{zL(-1)}Y(u,-z)w$ for $u\in V_{\bar 1}$ and $w\in M^i_{\bar r}.$ In this case, $c_{M^i_{\bar r},V_{\bar 1}}$ is a linear map
from $I\left(_{M^i_{\bar r},V_{\bar 1}}^{\ M^i_{\overline{r+1}}}\right)$ to $I\left(_{V_{\bar 1},M^i_{\bar r}}^{\ M^i_{\overline{r+1}}}\right)$ such that ${\cal Y}$ is mapped to ${\cal Y}'$ where
${\cal Y}'(u,z)w=e^{zL(-1)}{\cal Y}(w,-z)u$ for $u,w$ as before. Similarly, $c_{V_{\bar 1},M^i_{\bar r}}$ is a linear map from $I\left(_{V_{\bar 1},M^i_{\bar r}}^{\ M^i_{\overline{r+1}}}\right)$ to $I\left(_{M^i_{\bar r},V_{\bar 1}}^{\ M^i_{\overline{r+1}}}\right)$ such that $Y$ is mapped to ${\cal Y}.$ It is trivial to verify that ${\cal Y}'=Y$ and $c_{V_{\bar 1},M^i_{\bar r}}\circ c_{M^i_{\bar r},V_{\bar 1}}=\id_{M^i_{\bar r}\boxtimes V_{\bar 1}}.$

It remains to show that $V_{\bar s}$ for $s=0,1$ are the only simple objects in ${\cal Z}_2({\cal C}_{V_{\bar 0}}^0).$ Since ${\cal C}_{V_{\bar 0}}$ is modular, it follows from Theorem 3.2 of \cite{M} that
$$
\dim {\CC}_{V_{\bar 0}}=\dim C_{\CC_{V_{\bar 0}}}({\CC}_{V_{\bar 0}}^0) \cdot \dim {\cal C}_{V_{\bar 0}}^0.
$$
From the discussion in Section 8 we know that
$$\dim {\cal C}_{V_{\bar 0}}=2\dim {\cal C}_{V_{\bar 0}}^0.$$
This forces $\dim C_{{\cal C}_{V_{\bar 0}}}({\cal C}_{V_{\bar 0}}^0)=2.$ Clearly, $\ZZ_2(\CC_{V_{\bar 0}}^0)\subset  C_{\CC_{V_{\bar 0}}}(\CC_{V_{\bar 0}}^0).$ This implies that
$$
2\le \dim{\cal Z}_2({\cal C}_{V_{\bar 0}}^0) \le  \dim  C_{\CC_{V_{\bar 0}}}(\CC_{V_{\bar 0}}^0) =2,$$
and hence $\ZZ_2(\CC_{V_{\bar 0}}^0)=C_{\CC_{V_{\bar 0}}}(\CC_{V_{\bar 0}}^0)$.
\end{proof}

We can now formulate the 16-fold way conjecture in \cite{BGHNPRW} in the context of vertex operator algebra. Let $\BB$ be  a super-modular category. A modular category $\CC$, which contains $\BB$ as a full ribbon subcategory,  is called a \emph{minimal modular extension} or a \emph{modular closure} of ${\cal B}$ if $\dim{\CC}=2 \dim{\BB}.$

\begin{conj} If $\BB$ is a super-modular category, then $\BB$ admits a minimal modular extension. In this case, there are exactly 16 minimal modular extensions of $\BB$ up to braided monoidal equivalence.
\end{conj}

Under the assumption of the existence of minimal closure of super-modular category, the second part of the conjecture has been proved in \cite[Theorem 5.4]{LKW1}.

From Lemma \ref{l9.2} and its proof we immediately obtain:
\begin{thm}\label{t9.4} Let $V$ be a vertex operator superalgebra satisfying A1-A2. Then  ${\cal C}_{V_{\bar 0}}$ is a minimal modular extension of the super-modular category ${\cal C}_{V_{\bar 0}}^0.$
\end{thm}
In view of Theorem \ref{t9.4} and the 16-fold way conjecture, the following question arises:
\begin{q}
Does every super-modular category $\CC$ equivalent to $\CC^0_{V_{\bar 0}}$ for some vertex operator superalgebra $V$?
\end{q}
Our next  goal  is to construct a sequence $\{V^m\}_{m \ge 0}$ of vertex operator superalgebras from $V$ such that $\CC_{(V^m)_{\bar 0}}$ are minimal modular extension of $C_{V_{\bar 0}}^0$ and the equivalence classes of these modular categories $\CC_{(V^m)_{\bar 0}}$ are characterizes by the congruence class of $m$ modulo 16. We close this section with the following theorem.

\begin{thm}\label{t9.6} Let $V,U$ be vertex operator superalgebras satisfying A1-A2 and $U$ be holomorphic. Then ${\cal C}_{V_{\bar 0}}^0$ and ${\cal C}_{(U\otimes V)_{\bar 0}}^0$ are equivalent braided fusion categories. In particular, $\CC_{(U\otimes V)_{\bar 0}}$ is a minimal modular extension of $\CC^0_{V_{\bar 0}}$
\end{thm}
\begin{proof} Note that $(U \ot V)_{\bar 0}$ is an algebra object in $\CC_{U_{\bar 0} \ot V_{\bar 0}}$.
   Let $\BB = C_{\CC_{U_{\bar 0} \ot V_{\bar 0}}}((U \ot V)_{\bar 0})$, the M\"uger centralizer of $(U \ot V)_{\bar 0}$ in $\CC_{U_{\bar 0} \ot V_{\bar 0}}$. In view of \cite[Proposition 2.65]{CKM},  let $F_0: \BB \to \CC_{U \ot V}$ and $F_1 : \BB \to {\cal C}_{(U\otimes V)_{\bar 0}}$ be the induction functors, that means
 $$
 F_0(Y) = (U\ot V)\boxtimes_{U_{\bar 0} \otimes V_{\bar 0} } Y, \quad  F_1(Y) =(U\otimes V)_{\bar 0}\boxtimes_{U_{\bar 0} \otimes V_{\bar 0} } Y
 $$
 for $Y$ in $\BB$. By \cite[Theorem 2.67]{CKM}, $F_0, F_1$ are braided tensor functors. Since $F_1(Y)$ is a $(U\otimes V)_{\bar 0}$-submodule of $F_0(Y)$ and $\CC^0_{(U\otimes V)_{\bar 0}}$ is generated by the $(U\otimes V)_{\bar 0}$-submodules of super $U \ot V$-modules, $F_1(Y) \in \obj(\CC^0_{(U\otimes V)_{\bar 0}})$ for $Y \in \BB$.

 Since $(U\otimes V)_{\bar 0} = U_{\bar 0} \ot V_{\bar 0} \oplus U_{\bar 1} \ot V_{\bar 1}$, $U_{\bar 0} \ot X \in \BB$ for any object $X$ of $\CC_{V_{\bar 0}}^0$.
 Note that the functor $U_{\bar 0} \ot -: \CC_{V_{\bar 0}} \to \CC_{U_{\bar 0} \otimes V_{\bar 0}}$ is  a  faithfully full braided tensor functor, and  so is restriction $F_2:  \CC^0_{V_{\bar 0}} \to \BB$. Therefore, the composite functor $F=F_1 F_2: \CC^0_{V_{\bar 0}} \to {\CC}^0_{(U\otimes V)_{\bar 0}}$ is a braided tensor functor. Since $\CC_{V_{\bar 0}}^0$ is super-modular, $F$ is faithfully full by \cite[Corollary 3.26]{DMNO}.

 To show that $\CC^0_{V_{\bar 0}}$ is braided tensor equivalent to $\CC^0_{(U\ot V)_{\bar 0}}$, it suffices to show that every irreducible $(U\ot V)_{\bar 0}$-module is an image of $F$.

 Recall that the inequivalent irreducible super $V$-modules are $M^i$ with $i=0,\dots ,p.$ This implies that
 $\{M^i_{\bar r}\mid i=0,\dots, p, \text{and } r=0,1\}$ is a complete set of inequivalent irreducibles of $\CC_{V_{\bar 0}}^0$. Moreover, inequivalent irreducible $U\otimes V$-modules are $U\otimes M^i.$ Therefore, $(U\otimes M^i)_{\bar r}=
 U_{\bar 0}\otimes M^i_{\bar r}+U_{\bar 1}\otimes M^i_{\overline{1-r}} $ for $i=0,\dots ,p$ and $r=0,1$ are  all the inequivalent irreducible $(U\otimes V)_{\bar 0}$-modules of $\CC^0_{(U \ot V)_{\bar 0} }$.

 For any irreducible $X\in\obj({\cal C}_{V_{\bar 0}}^0)$,
 $$F(X)=(U\otimes V)_{\bar 0}\boxtimes (U_{\bar 0}\otimes X)$$
 which is isomorphic to $U_{\bar 0}\otimes X+ U_{\bar 1}\otimes(V_{\bar 1}\boxtimes X)$
 as $V_{\bar 0}\otimes U_{\bar 0}$-modules.  Therefore, $F(X)$ is the irreducible $(V\otimes U)_{\bar 0}$-module which contains an irreducible $U_{\bar 0}\otimes V_{\bar 0}$-submodule isomorphic to $U_{\bar 0}\ot X$.  Therefore,
By the same reason
 $$
 F(M^i_{\bar r}) \cong (U \ot M^i)_{\bar r}
 $$
 as $(U \ot V)_{\bar 0}$-modules for $i =1,\dots, p$ and $r=0,1$.  Thus, $F: \CC^0_{V_{\bar 0}} \to \CC^0_{(U \ot V)_{\bar 0}}$ is an equivalence.  The last statement follows immediately from Theorem \ref{t9.4}.
\end{proof}

The Gauss sum $\tau_1(\CC)$ of a ribbon fusion category $\CC$ is defined as
$$
\tau_1(\CC) = \sum_{X \in \irr(\CC)} \qdim(X)^2 \cdot \theta_X
$$
where $\qdim(X)$ is the pivotal (or quantum) dimension of the simple object $X$ and $\theta_X$ denotes the scalar of the twist. The Gauss sums and their higher degree generalizations $\tau_n(\CC)$ are invariants of ribbon fusion categories (cf. \cite{NSW}). In the case of a fermionic modular category, we follow some idea in \cite{BGNPRW} to prove that the centralizer of the fermion has zero contribution to the Gauss sum.
\begin{lem} \label{l:Gauss_sum}
Let $\CC$ be a pseudounitary modular tensor category over $\C$, $f$ a fermion of $\CC$, and $\CC^0$ the M\"uger centralizer of $f$. Then
$$
\tau_1(\CC^0) = 0
$$
\end{lem}
\begin{proof}
Let $X \in \irr(\CC^0)$. Then $S_{X, f} =\qdim(X)$ where $S_{X,Y} $ is the trace of $c_{Y, X} \circ c_{X, Y}$. Since $f$ is an invertible object, $X \ot f$ is a simple object of $C^0$ and $\qdim(X \ot f) = \qdim(X)$.
 On the other hand, by the twist equation, we have
$$
\theta_{X \ot f} \qdim(X \ot f) = \theta_X \theta_f S_{X,f},
$$
which implies $\theta_{X \ot f} = - \theta_X$. In particular, the action of $f$ on $\irr(\CC)$ has no fixed point. Therefore, there exists a subset $\OO$ of $\irr(\CC^0)$ such that $\bigcup_{X \in \OO} \{X, X \ot f\} = \irr(\CC^0)$. Thus,
$$
\tau_1(\CC^0)= \sum_{X \in \OO} \qdim(X)^2\cdot \theta_X - \qdim(X \ot f)^2\cdot \theta_X = 0\,.\quad \qedhere
$$
\end{proof}

\section{16 minimal modular extensions}

In this section we use the holomorphic vertex operator superalgebas $V(l,\Z+\frac{1}{2})$ for $l\geq 1$ and Theorem \ref{t9.6} to obtain all the 16 minimal modular extensions of ${\cal C}_{V_{\bar 0}}^0$ for any given vertex operator superalgebra satisfying A1-A2.

The construction of  $V(l,\Z+\frac{1}{2})$ is well known (see \cite{FFR}, \cite{KW}, \cite{L1}). Let $H_l=\oplus_{i=1}^{l}{\Bbb C}a_{i}$ be a complex vector space equipped
with a nondegenerate symmetric bilinear form $(\cdot,\cdot)$ such that  $(a_i,a_j)=2\delta_{i,j}.$  Let $A(l,{\Bbb Z}+\frac{1}{2})$ be the
associative algebra generated by $\{a(n)\,|\,a\in H_l,n\in {\Bbb
Z}+\frac{1}{2}\}$ subject to the relation
$$[a(n),b(m)]_{+}=(a,b)\delta_{{m+n},0}.$$
Let $A(l,{\Bbb
Z}+ \frac{1}{2})^+$  be the subalgebra generated by $\{a(n)\,|\,a\in H_l,n\in
{\Bbb Z}+\frac{1}{2},n>0 \},$ and make ${\Bbb C}$ a $1$-dimensional
 $A(l,{\Bbb
Z}+ \frac{1}{2})^+$-module so that $a_i(n)1=0$
for $n>0.$ The induced module
\begin{eqnarray*}
& &V(l,{\Bbb Z}+\frac{1}{2})=A(l,{\Bbb Z}+\frac{1}{2})\otimes_{A(l,{\Bbb
Z}+ \frac{1}{2})^+}\C\\
& & \ \ \ \ \ \ \ \ \ \  \ \ \ \ \  \  \ \cong \bigwedge [a_i(-n)|n >0,n\in {\Bbb Z}+\frac{1}{2} ,i=1,2,\dots l]\ ({{\rm linearly}}).
 \end{eqnarray*}
is a holomorphic vertex operator superalgebra generated by $a_i(-\frac{1}{2})$ for $i=1,\dots ,l$ and $Y(a_i(-\frac{1}{2}),z)=a_i(z)=\sum_{n\in\Z}a_i(-n-\frac{1}{2})z^{-n-1}.$ For example, if $l=1$ then $V(1,{\Bbb Z}+\frac{1}{2})$ is isomorphic to $L(\frac{1}{2},0)+L(\frac{1}{2},\frac{1}{2})$ as a module for the Virasoro vertex operator algebra $L(\frac{1}{2},0).$ Moreover, $V(1,\Z+\frac{1}{2})_{\bar 0}=L(\frac{1}{2},0)$ and
$V(1,\Z+\frac{1}{2})_{\bar 1}=L(\frac{1}{2},\frac{1}{2}).$ If $l=2k$ is even then $V(l,\Z+\frac{1}{2})$ is isomorphic to the lattice vertex operator superalgebra $V_{\Z^k}$ where $\Z^k$ is the lattice in $\R^k$ with the standard inner product.

As usual, we use $\sigma$ to denote the canonical automorphism of $V(l,\Z+\frac{1}{2}).$  To construct $\sigma$-twisted $V(l,\Z+\frac{1}{2})$-modules we need to consider two cases $l$ is even or odd. If $l=2k$ is even,
The $H_{2k}$ can be written as
$$H_{2k}=\sum_{i=1}^{k}{\Bbb C} b_i+\sum_{i=1}^{k}{\Bbb C}b_i^* $$ with
$(b_i,b_j)=(b_i^*,b_j^*)=0, (b_i,b_j^*)=\delta_{i,j}.$ Let $A(2k,{\Bbb
Z})$ be the associative algebra generated by $\{b(n)\,|\,b\in H_{2k}, n\in
{\Bbb Z}\} $ subject to the relation
$$[a(m),b(n)]_+=(a,b)\delta _{m+n,0}$$ Let $A(2k,{\Bbb Z})^+$ be the
subalgebra generated by $\{b_i(n),b_i^{*}(m)\,|\,n\geq 0, m >0,
i=1,\ldots, k\},$ and make ${\Bbb C}$  a $1$-dimensional $A(2k,{\Bbb
Z})^+$-module with $b_i(n)1=0$  and $b_i^*(m)1=0$ for
 $n\geq 0,$ $m>0,$  $i=1,\dots , k.$ Consider the induced $A(2k,{\Bbb Z})$-module
\begin{eqnarray*}
 V(2k,{\Bbb Z})=A(2k,{\Bbb Z})\otimes_{A(2k,{\Bbb Z})^+} {\Bbb C} \cong \Lambda [b_i(-n),b_i^{*}(-m)\,|\,n,m\in {\Bbb Z},n>0, m\geq0].
\end{eqnarray*}
By Proposition 4.3 in \cite{L2}, $V(2k,{\Bbb Z})$ is an irreducible
 $\sigma$-twisted $V(2k,{\Bbb Z}+\frac{1}{2})$-module such that
\begin {eqnarray*}
Y_{V(2k,{\Bbb Z})}(u(-\frac{1}{2}),z)=u(z)=\sum_{n\in {\Bbb Z}}u(n)z^{-n-1/2}
\end{eqnarray*}
for $u\in H_{2k}.$ Moreover,  $V(2k,{\Bbb Z})$ is the only irreducible
 $\sigma$-twisted $V(2k,{\Bbb Z}+\frac{1}{2})$-module up to isomorphism \cite{DZ2}. As a result,  $V(2k,{\Bbb Z}+\frac{1}{2})_{\bar 0}$ has 4 inequivalent irreducible modules  $V(2k,{\Bbb Z}+\frac{1}{2})_{\bar r},$
  and $V(2k,{\Bbb Z})_{\bar r}$ ($r=0,1$) of weights $0,\frac{1}{2}, \frac{k}{8}, \frac{k}{8}$, and quantum dimension 1.

If $l=2k+1$ is odd,   $H_{2k+1}$ can be decomposed into:
$$H_{2k+1}=\sum_{i=1}^{k}{\Bbb C} b_i+\sum_{i=1}^{k}{\Bbb C}b_i^*+{\Bbb C}e$$
with $(b_i,b_j)=(b_i^*,b_j^*)=0,  (b_i,b_j^*)=\delta_{i,j}, (e,b_i)=(e,b_i^*)=0,(e,e)=2.$
Let $A(2k+1,{\Bbb Z})$ be the associative algebra generated by $a(n)$ for $a\in H_{2k+1}$ and $n\in \Z$ subject to the same relation as before, Let $A(2k+1,{\Bbb Z})^+$ be the subalgebra generated by
$$\{b_i(n),b_i^*(m),e(m)\,|\,m,n\in \Z,n\geq 0, m>0,i=1,\ldots, k\}$$
and
make ${\Bbb C}$ a $1$-dimensional $A(2k+1,{\Bbb Z})^+$-module with
$b_i(n)1=0$ for $n\geq0 $ and $b_i^*(m)1=e(m)1=0$ for $m>0,$
$i=1,\dots , k.$ Set
$$
 V(2k+1,{\Bbb Z})=A(2k+1,{\Bbb Z})\otimes_{A(2k+1,{\Bbb Z})^+} {\Bbb C}\,.
$$
It is easy to see that $V(2k+1,\Z)$ is isomorphic to the exterior algebra
$$ W(2k+1,{\Z})=\Lambda [b_i(-n),b_i^{*}(-m),e(-m)\,|\,n,m\in {\Z},n>0, m\geq0]\,.$$
as vector spaces.
    Let $W(2k+1,{\Z})=W(2k+1,{\Z})_{\bar 0}\oplus W(2k+1,\Z)_{\bar 1}$
be the decomposition into the even and odd parity subspaces, and
$$V_{\pm}(2k+1,\Z)=(1\pm e(0))W(2k+1,{\Bbb Z})_{\bar 0}\oplus (1\mp e(0))W(2k+1,{\Bbb Z})_{\bar 1}.$$

Then
$$V(2k+1,\Z)=V_+(2k+1,\Z)\oplus V_-(2k+1,\Z)$$
and $V_{\pm}(2k+1,\Z)$ are irreducible $A(2k+1,{\Bbb Z})$-modules.
It follows from  Proposition 4.3 in \cite{L2} that $V_{\pm}(2k+1,\Z)$
are irreducible $\sigma$-twisted  modules for $V(2k+1,{\Bbb Z+\frac{1}{2}})$ so
that
\begin {eqnarray*}
Y_{V(2k+1,{\Bbb Z})}(u(-\frac{1}{2}),z)=u(z)=\sum_{n\in {\Bbb Z}}u(n)z^{-n-1/2}
\end{eqnarray*}
for $u\in H_{2k+1}.$ Moreover, $V_{\pm}(2k+1,\Z)$ are the only inequivalent irreducible $\sigma$-twisted modules and are isomorphic irreducible $V(2k+1,{\Bbb Z}+\frac{1}{2})_{\bar 0}$-modules \cite{DZ2}.
 In this case $V(2k+1,{\Bbb Z}+\frac{1}{2})_{\bar 0}$ has 3 inequivalent irreducible modules $V(2k+1,{\Bbb Z}+\frac{1}{2})_{\bar r}$ for $r=0,1$ and $V_{+}(2k+1,\Z)$ of weights
 $0, \frac{1}{2}$ and $\frac{2k+1}{16}$, and quantum dimensions $1,1$ and $\sqrt{2}$.

Let $V$ be a vertex operator superalgebra satisfying A1-A2. Set $V^0=V$ and $V^l=V(l,\Z+\frac{1}{2})\otimes V$ for $l\geq 1.$ According to Theorem \ref{t9.4}, ${\cal C}_{(V^l)_{\bar 0}}$ is a minimal modular extension of  ${\cal C}_{V_{\bar 0}}^0$ for $l\geq 0.$  We denote the Virasoro vector of $V^l$ by $\omega^l$ for $l\geq 1$ and write $Y(\omega^l,z)=\sum_{n\in\Z}L^l(n)z^{-n-2}.$ Let $T_l$ be the corresponding $T$-matrix associated to $(V^l)_{\bar 0}$ and set ${\frak t}_l=e^{2\pi i(c+\frac{l}{2})/24}T_l$ which is the matrix for the operator $e^{2\pi iL^l(0)}$
acting on the inequivalent irreducible  $(V^l)_{\bar 0}$-modules. Then ${\frak t}_l$ is the T-matrix of the modular tensor category ${\cal C}_{(V^l)_{\bar 0}}.$

The following result is an immediate consequence of Theorem \ref{twist}.
\begin{lem} \label{l:simple_s_modules} The inequivalent irreducible  $\sigma$-twisted $V^l$-modules are
$$\{V(l,\Z)\otimes N^j , (V(l,\Z)\otimes N^j)\circ\sigma, V(l,\Z)\otimes N^k\,|\,j=0,\dots ,q, k=q+1,\dots ,p\}$$
if $l$ is even, and
$$\{N^{l,j},   N^{l,k}, N^{l,k}\circ\sigma\,|\,j=0,\dots ,q, k=q+1,\dots ,p\}$$
if $l$ is odd where $V(l,\Z)\otimes (N^j+N^j\circ \sigma)=2N^{l,j}$ and  $V(l,\Z)\otimes N^k=N^{l,k}\oplus N^{l,k}\circ \sigma.$
\end{lem}

\begin{coro}\label{ct} The inequivalent simple $\Vl0$-modules from the $\sigma$-twisted $V^l$-modules are
$$\{V(l,\Z)\otimes N^j, (V(l,\Z)\otimes N^k)_{\bar r}\,|\,j=0,\dots ,q, k=q+1,\dots ,p, r=0,1\}$$
if $l$ is even. In this case,
$$
\qdim_{(V^l)_{\bar 0 }}(V(l,\Z)\ot N^j) = \qdim_{V_{\bar{0}}} (N^j), \quad \qdim_{(V^l)_{\bar{0}}}(V(l,\Z)\ot N^k)_{\bar{r}}=
\qdim_{V_{\bar{0}}} (N^k_{\bar r})
$$
for $j=0,\dots ,q, k=q+1,\dots ,p, r=0,1$.

If $l$ is odd, the inequivalent simple $(V^l)_{\bar 0}$-modules from the $\sigma$-twisted $V^l$-modules are
$$\{N^{l,j}_{\bar r},  N^{l,k}\,|\,j=0,\dots ,q, k=q+1,\dots ,p,r=0,1\}\,$$  and
$$
\qdim_{\Vl0} ( N_{\bar{r}}^{l, j}) = \frac{1}{\sqrt{2}}\cdot \qdim_{V_{\bar 0}} (N^j), \quad \qdim_{\Vl0}(N^{l, k}) =
\sqrt{2} \cdot \qdim_{V_{\bar 0}} (N^k_{\bar r})
$$
for $j=0,\dots ,q, k=q+1,\dots ,p, r=0,1$.
\end{coro}
\begin{proof}
The set of simple  $\Vl0$-modules from the $\s$-twisted $V^l$-modules follows immediately from Lemma \ref{l:simple_s_modules} for any nonzero integer $l$.  Let us denote $U^l= \VlhZ$. If $l$ is even, then
$\qdim_{U^l}(\VlZ)=\qdim_{U^l_{\bar 0}} (\VlZ)=1$. For $j=0, \dots, p$, $\VlZ \ot N^j$ is an unstable $\s$-twisted $V^l$-module. It follows from Proposition \ref{MTH} that
$$
\qdim_{\Vl0}(\VlZ \ot N^j) = 2 \qdim_{V^l}(\VlZ \ot N^j ) =2 \qdim_V(N^j) = \qdim_{V_{\bar 0}}(N^j)\,.
$$
For $k = p+1, \dots, q$, $\VlZ \ot N^k$ is  $\s$-stable. By Proposition \ref{MTH},
$$
\qdim_{\Vl0}(\VlZ \ot N^k)_{\bar r} = \qdim_{V^l}(\VlZ \ot N^k) = \qdim_V(N^k) = \qdim_{V_{\bar 0}}(N^k_{\bar r})
$$
for $r=0,1$.

If $l$ is odd, then $\qdim_{U^l}(\VlZ)=\qdim_{U^l_{\bar 0}} (\VlZpm)=\sqrt{2}$. For $j=0, \dots, p$, $N^{l,j}$ is a $\s$-stable $\s$-twisted $V^l$-module and
$$
\qdim_{V^l}(\VlZ \ot (N^j\oplus N_\s^j)) = 2\qdim_{V^l}(N^{l,j}) = 2 \qdim_{\Vl0}(N_{\bar r}^{l,j} )
$$
for any $r=0,1$. On the other hand,
$$
\qdim_{V^l}(\VlZ \ot (N^j\oplus N_\s^j)) = \sqrt{2} \cdot \qdim_V(N^j\oplus N_\s^j) =\sqrt{2} \qdim_{V_{\bar 0}}(N^j)\,.
$$
Thus, we have
$$
\qdim_{\Vl0}(N_{\bar r}^{l,j} )  = \frac{1}{\sqrt{2}}\qdim_{V_{\bar 0}}(N^j)
$$
for $r=0,1$. Similarly, For $k=p+1, \dots, q$, $N^{l,k}$ is a $\s$-unstable $\s$-twisted $V^l$-module and
\begin{align*}
\qdim_{\Vl0}(N^{l,k})  & = 2\qdim_{V^l}(N^{l,k}) = \qdim_{V^l}(N^{l,k} \oplus N^{l,k} _\s) =   \qdim_{V^l}(\VlZ\ot N^k)\\
& =\sqrt{2}\cdot \qdim_{V}(N^k) = \sqrt{2} \cdot \qdim_{V_{\bar 0}}(N_{\bar r}^k)
\end{align*}
for $r=0,1$.
\end{proof}

\begin{thm} \label{t10.3}  The minimal modular extensions  ${\cal C}_{(V^l)_{\bar 0}},$  ${\cal C}_{(V^m)_{\bar 0}}$ of  ${\cal C}_{V_{\bar 0}}^0$ are braided equivalent if and only if $l$ and $m$ are congruent modulo 16. In particular, we obtain 16 minimal modular extensions of  ${\cal C}_{V_{\bar 0}}^0$.
\end{thm}
\begin{proof} Since $\CC_{(V^m)_{\bar{0}}}$ has positive quantum dimensions, its spherical pivotal structure is uniquely determined by the fusion category  $\CC_{(V^m)_{\bar{0}}}$. Therefore, $\CC_{(V^l)_{\bar{0}}}$, $\CC_{(V^m)_{\bar{0}}}$ are equivalent braided fusion categories if and only if they are equivalent modular categories. The later implies they have the same Gauss sums. Therefore, we proceed to compute the Gauss sum $\tau_1(\CC_{\Vl0})$. It follows from Lemma \ref{l:Gauss_sum} that
$$
\tau_1(\CC_{\Vl0}) = \sum_{X \in {\irr}(\CC_{\Vl0}^1)} \qdim(X)^2 \cdot \theta_X
$$
where ${\irr}(\CC^1_{\Vl0})$ is the set of inequivalent simple $\Vl0$-modules from the $\s$-twisted $V^l$-modules.

By Corollary \ref{ct}, the  inequivalent irreducible $(V^{l})_{\bar 0}$-modules from the $\sigma$-twisted modules are
 $$\{\VlZ \ot N^j,  \VlZ_{\bar 0}\ot N^k_{\bar r}+ \VlZ_{\bar 1}\ot N^k_{\ol {1-r}}\,|\,j=0,\dots ,q, k=q+1,\dots ,p, r
 =0,1\}$$
 if $l$ is even. The actions of $e^{2\pi i L^l(0)}$ on $\VlZ \ot N^j$ and $\VlZ_{\bar 0} \otimes N^k_{\bar r}+ \VlZ_{\bar 1}\ot N^k_{\ol {1-r}}$ are respectively are $e^{2\pi i(\lambda_{N^j}+\frac{l}{16})}$ and $e^{2\pi i(\lambda_{N^k}+\frac{l}{16})}$ for $r=0,1$, where $\lambda_{N^j}$ is the weight of $N^j$.  Therefore,
 \begin{align*}
    \tau_1(\CC_{\Vl0}) & = \sum_{j=0}^p  \qdim_{\Vl0}(\VlZ \ot N^j)^2 \cdot e^{2\pi i(\lambda_{N^j}+\frac{l}{16})} \\ &+\sum_{r=0}^1\sum_{k=p+1}^q  \qdim_{\Vl0}(\VlZ \ot N^k)_{\bar r}^2 \cdot e^{2\pi i(\lambda_{N^k}+\frac{l}{16})} \\
    & =e^{\frac{2\pi il}{16}} \left(\sum_{j=0}^p  \qdim_{V_{\bar 0}}(N^j)^2 \, e^{2\pi i \lambda_{N^j}} +\sum_{r=0}^1\sum_{k=p+1}^q  \qdim_{V_{\bar 0}}(N^k_{\bar r})^2 \, e^{2\pi i \lambda_{N^k}}\right)\\
    & = e^{\frac{2\pi il}{16}} \tau_1(\CC_{V_{\bar 0}})\,.
 \end{align*}
Again by Corollary \ref{ct}, the  inequivalent irreducible $(V^{l})_{\bar 0}$-modules from the $\sigma$-twisted modules are
$$\{N^{l,j}_{\bar r},  N^{l,k}|j=0,\dots ,q, k=q+1,\dots ,p,r=0,1\}$$
if $l$ is odd. The actions of $e^{2\pi i L^l(0)}$ on $N^{l,j}_{\bar r}$ is $e^{2\pi i(\lambda_{N^j}+\frac{l}{16})}$ and on $N^{l,k}$ is $e^{2\pi i(\lambda_{N^k}+\frac{l}{16})}$. Thus, for $r=0,1$,

\begin{align*}
    \tau_1(\CC_{\Vl0}) & = \sum_{r=0}^1 \sum_{j=0}^p  \qdim_{\Vl0}(N^{l,j}_{\bar r})^2 \, e^{2\pi i(\lambda_{N^j}+\frac{l}{16})} + \sum_{k=p+1}^q  \qdim_{\Vl0}(N^{l,k})^2 \, e^{2\pi i(\lambda_{N^k}+\frac{l}{16})} \\
    & =e^{\frac{2\pi il}{16}} \left(\sum_{r=0}^1 \sum_{j=0}^p  \frac{1}{2}\qdim_{\Vl0}(N^j)^2 \, e^{2\pi i\lambda_{N^j}} + \sum_{k=p+1}^q  2\qdim_{V_{\bar 0}}(N^k_{\bar r})^2 \, e^{2\pi i\lambda_{N^k}}\right) \\
    & =e^{\frac{2\pi il}{16}} \left(\sum_{j=0}^p \qdim_{\Vl0}(N^j)^2 \, e^{2\pi i\lambda_{N^j}} + \sum_{r=0}^1 \sum_{k=p+1}^q  \qdim_{V_{\bar 0}}(N^k_{\bar r})^2 \, e^{2\pi i\lambda_{N^k}}\right) \\
     &= e^{\frac{2\pi i l}{16}} \tau_1(\CC_{V_{\bar 0}})\,.
 \end{align*}

Therefore,  $\tau_1(\CC_{\Vl0}) = e^{\frac{2\pi i}{16}} \tau_1(\CC_{V_{\bar 0}})$ for any integer $l \ge 0$. As a result, $\tau_1(\CC_{\Vl0}) = \tau_1(\CC_{(V^m)_{\bar 0}})$ if and only if  $l \equiv m$ modulo $16$, and  there are at least 16 inequivalent modular categories which are minimal extensions of $\CC^0_{V_{\bar 0}}$. By  [LKW, Theorem 5.4], $\CC^0_{V_{\bar 0}}$ has exactly 16 minimal extensions. Thus, $\CC_{(V^l)_{\bar 0}}$ and $\CC_{(V^m)_{\bar 0}}$ are equivalent minimal extensions of $\CC^0_{V_{\bar 0}}$ if and only if $l \equiv m$ modulo $16$. These 16 minimal extensions of are also inequivalent as braided fusion categories as they have distinct Gauss sums.
\end{proof}


\begin{thebibliography}{FLMM}

\bibitem[ABD]{ABD}
T. Abe, G. Buhl and C. Dong, Rationality, Regularity, and $C_2$-cofiniteness,
{\em Trans. Amer. Math. Soc.} {\bf 356} (2004), 3391-3402.

\bibitem[ADJR]{ADJR}C. Ai, C. Dong, X. Jiao and L. Ren, The irreducible modules and fusion rules for the parafermion vertex operator algebras,   {\em J. Math. Phys.} {\bf 58} (2017), no. 4, 041704, 31 pp.

\bibitem[BDM]{BDM}K. Barron, C. Dong and G. Mason,
Twisted sectors for tensor product VOAs associated to permutation
groups {\em Comm. Math. Phys.} {\bf 227} (2002), 349-384.

\bibitem[BGH]{BGHNPRW} P. Bruillard, C. Galindo, T. Hagge, S. Ng, J. Plavnik, E. Rowell, Z. Wang, Fermionic modular categories and the 16-fold way, \emph{J. Math. Phys.} \textbf{58} (2017), no. 4, 041704, 31 pp.
\bibitem[BGN]{BGNPRW} P. Bruillard. C. Galindo. S. Ng, J. Plavnik, E. Rowell, Z. Wang, Classification of Super-Modular Categories by Rank,  \emph{Algebr Represent Theor }(2019). https://doi.org/10.1007/s10468-019-09873-9.

\bibitem[CM]{CM} S. Carnahan and M. Miyamoto, Regularity of fixed-point vertex operator subalgebras, arXiv:1603.05645.

\bibitem[CKL]{CKL}T. Creutzig, S. Kanade and A. Linshaw, Simple current extensions beyond semi-simplicity, {\em  Commun. Contemp. Math.} https://doi.org/10.1142/S0219199719500019.
\bibitem[CKM]{CKM} T. Creutzig, S. Kanade and R. McRae, Tensor categories for vertex operator superalgebra extensions,  arXiv:1705.05017.

\bibitem[C]{C} S. B. Conlon Twisted group algebras and their representations, {\em J. Austral. Math. Soc.} {\bf 4 } (1964), 152-173.
\bibitem[DMNO]{DMNO}A. Davydov, M. M\"{u}ger, D. Nikshych and V. Ostrik,  The Witt group of non-degenerate braided fusion categories, {\em J. Reine Angew. Math.} {\bf 677} (2013), 135--177.
\bibitem[De]{De} P. Deligne,
Catégories tannakiennes, {\em The Grothendieck Festschrift}, Vol. II, 111--195,
Progr. Math., 87, Birkhäuser Boston, Boston, MA, 1990.
\bibitem [DPR]{DPR} R.  Dijkgraaf, V. Pasquier and P. Roche, Quasi Hopf algebras, group
cohomology and orbifold models, {\em Nuclear Physics B} (Proc. Suppl.) \textbf{18B}
(1990), 60-72.

\bibitem[D1]{D1} C. Dong, Vertex algebras associated with even lattices, {\em J. Alg.} {\bf 161} (1993), 245-265.
\bibitem[D2]{D2}C. Dong, Twisted modules for vertex algebras associated with
even lattices, {\it J. Algebra} {\bf 165} (1994), 90-112.

\bibitem[DH]{DH} C. Dong and J. Han, On rationality of vertex operator superalgebras, {\em International Math. Research
Notices} {\bf 2013} (2013), Article ID 80468, 15 pages.

\bibitem[DJX]{DJX} C. Dong, X. Jiao and F. Xu, Quantum Dimensions and Quantum Galois Theory, {\em Trans. AMS.} {\bf 365} (2013), 6441-6469.

\bibitem[DL1]{DL1} C. Dong and J. Lepowsky, Generalized Vertex
Algebras and Relative Vertex Operators, {\em Progress in Math,} {\bf Vol. 112}, Birkh\"{a}user, Boston, 1993.

\bibitem[DL2]{DL2} C. Dong and J. Lepowsky, The algebraic structure of
relative twisted  vertex operators, {\em J. Pure and Applied
Algebra} {\bf 110} (1996), 259-295.
\bibitem[DLM1]{DLM1}  C. Dong, H. Li and G. Mason, Simple currents and extensions
of vertex operator algebras, {\em Comm. Math. Phys.}
{\bf 180} (1996), 671-707.

\bibitem[DLM2]{DLM2} C. Dong, H. Li and G. Mason,
Compact automorphism groups of \voas , {\em Int. Math. Res. Not.} {\bf 18}  (1996), 913-921.

\bibitem[DLM3]{DLM3} C. Dong, H. Li and G. Mason,
Regularity of rational vertex operator algebras, {\em  Adv. Math.} {\bf 132} (1997), 148-166.

\bibitem[DLM4]{DLM4} C. Dong, H. Li and G. Mason,
Twisted representations of vertex operator algebras, {\em Math. Ann.}
{\bf  310} (1998), 571--600.

\bibitem[DLM5]{DLM5} C. Dong, H. Li and G. Mason,
Vertex operator algebras and associative algebras, {\em  J. Alg.}
{\bf 206} (1998), 67-96.

\bibitem[DLM6]{DLM6} C. Dong, H. Li and G. Mason, Twisted representations of vertex operator algebras and associative algebras, {\em Int. Math. Res. Not.}  {\bf 8} (1998), 389-397.


\bibitem[DLM7]{DLM7} C. Dong, H. Li and G. Mason,
Modular-invariance of trace functions in orbifold theory and generalized moonshine, {\em  Comm. Math. Phys.}
{\bf 214} (2000), 1-56.

\bibitem[DLN]{DLN} C. Dong, X. Lin, S. Ng, Congruence property in conformal field theory, {\em Algebra \& Number Theory,} {\bf 9} (2015), 2121-2166.

\bibitem[DM]{DM}C. Dong and G. Mason, On quantum Galois theory, {\em Duke
Math. J.} {\bf 86} (1997), 305-321.

\bibitem[DMZ]{DMZ} C. Dong, G. Mason and Y. Zhu, Discrete series of the
Virasoro algebra and the moonshine module, {\em Proc. Symp. Pure. Math., American Math. Soc.} {\bf 56} II (1994), 295-316.


\bibitem[DRX]{DRX} C. Dong, L. Ren, F. Xu, On orbifold Theory, {\em Adv. Math.} {\bf 321} (2017), 1-30.



\bibitem[DY]{DY} C. Dong and G. Yamskulna,
Vertex operator algebras, Generalized double and dual pairs,
{\em Math. Z.} {\bf 241} (2002), 397-423.

\bibitem[DYu]{DYu}C. Dong and N. Yu, $Z$-graded weak modules and regularity, {\em Comm. Math. Phys.} {\bf 316} (2012), 269-277.
\bibitem[DZ1]{DZ1} C. Dong and Z. Zhao,  Modularity in orbifold theory for vertex operator
superalgebras, {\em Comm. Math. Phys.} {\bf 260} (2005), 227-256.


\bibitem[DZ2]{DZ2}  C. Dong and Z. Zhao, Twisted representations of vertex
operator superalgebras, {\em Comm. Contemp. Math.} {\bf 8} (2006),
101-122.

\bibitem[DR]{DR} S. Doplicher and J. E. Roberts, A new duality theory for compact groups, {\em Invent. Math.} {\bf 98} (1989), 157--218.
\bibitem[ENO]{ENO} P. Etingof, D. Nikshych and V.  Ostrik,  On fusion categories,  {\em Ann. Math. (2)} {\bf 162} (2005), 581-642.

\bibitem[FFR]{FFR}
Alex J. Feingold, Igor B. Frenkel and John F. X. Ries, Spinor
Construction of Vertex Operator Algebras, Triality, and
$E_{8}^{(1)}$, {\em Contemporary Math.} {\bf 121}, 1991.
\bibitem[FHL]{FHL}
I. Frenkel, Y. Huang and J. Lepowsky, On axiomatic approaches to vertex operator algebras and modules, {\em Mem. Amer. Math. Soc.} {\bf 104} 1993.

\bibitem[FLM1]{FLM1}  I. B. Frenkel, J. Lepowsky and A. Meurman,
Vertex operator calculus, in: {\em Mathematical Aspects of String
Theory, Proc. 1986 Conference, San Diego.} ed. by S.-T. Yau, World
Scientific, Singapore, 1987, 150-188.

\bibitem[FLM2]{FLM2} I. Frenkel, J. Lepowsky and A. Meurman,
Vertex Operator Algebras and the Monster, {\em Pure and Applied
Math.,} {\bf Vol. 134,} Academic Press, Boston, 1988.

\bibitem[HA]{HA} J. Han and C. Ai, Three equivalent rationalities of vertex operator superalgebras, {\em J. Math. Phys.} {\bf 56} (2015), no. 11, 111701, 7 pp.
\bibitem[HMT]{HMT}
 A. Hanaki, M. Miyamoto, and D. Tambara, Quantum Galois theory for finite groups, {\em Duke Math. J.} {\bf 97} (1999), 541--544.



 \bibitem[Hu1]{Hu1}
Y.-Z. Huang, Rigidity and modularity of vertex tensor categories, \emph{Commun. Contemp. Math.} {\bf 10}, suppl. 1 (2008), 871–-911.
\bibitem[Hu2]{Hu2}
 Y.-Z. Huang, Vertex operator algebras and the Verlinde Conjecture, {\it  Comm. Contemp. Math.} {\bf 10} (2008), 103--154.

\bibitem[HKL]{HKL} Y. Huang, A. Kirillov Jr. and J. Lepowsky, Braided tensor categories and extensions of vertex operator algebras,  {\em Comm. Math. Phys.} {\bf 337} (2015),  1143-1159.
\bibitem[HL1]{HL1} Y. Huang and J. Lepowsky, A theory of tensor
products for module categories for a vertex operator algebra, I, {Selecta. Math. (N. S)} {\bf 1} (1995), 699--756.
\bibitem[HL2]{HL2} Y. Huang and J. Lepowsky, A theory of tensor
products for module categories for a vertex operator algebra, II, {Selecta. Math. (N. S)} {\bf 1} (1995), 756--786.
\bibitem[HL3]{HL3} Y. Huang and J. Lepowsky, A theory of tensor
products for module categories for a vertex operator algebra, III. {J. Pure Appl. Alg.} {\bf 100} (1995),
141--171.




\bibitem[KW]{KW}
V. Kac and W. Wang, Vertex operator superalgebras and
representations, {\em Contem. Math., AMS} {\bf Vol. 175} (1994), 161-191.
\bibitem[Kl]{Kl} A Kleshchev, Linear and Projective Representations of Symmetric Groups. {\em Cambridge Tracts in Mathematics} {\bf Vol. 163},  Cambridge University Press 2005, New York.
\bibitem[Ki]{Ki} A. Kitaev, Anyons in an exactly solved model and beyond, {\em Ann. Phys.} {\bf 321}  (2006), 2--111.
\bibitem[KO]{KO} A. Kirillov Jr. and V. Ostrik,  On a q-analogue of the McKay correspondence and the ADE classification of $sl_2$ conformal field theories, {\em Adv. Math.} {\bf 171} (2002), 183-227.
\bibitem[LKW1]{LKW1} T. Lan, L. Kong, X.-G. Wen,
Modular extensions of unitary braided fusion categories and
              {$2+1{\rm D}$} topological/{SPT} orders with symmetries,
              {\em Comm. Math. Phys.} {\bf 351}  (2017), 709--739.
\bibitem[LKW2]{LKW2} T. Lan, L. Kong, X.-G. Wen,
Classification of (2+1)-dimensional topological order and symmetry-protected topological order for bosonic and fermionic systems with on-site symmetries, {\em Phys. Rev. B} {\bf 95} (2017), 235140.
\bibitem[L1]{L1}
 H. Li, Local systems of vertex operators, vertex superalgebras and modules, {\em J. Pure Appl. Algebra} {\bf 109} (1996), 143--195.

\bibitem[L2]{L2}
H. Li, Local systems of twisted vertex operators, vertex operator superalgebras and twisted modules, {\em  Contemp. Math. AMS.}
{\bf  193} (1996), 203-236.

\bibitem[L3]{L3}  H. Li,  \emph{Some finiteness properties of regular vertex
operator algebras}, J. Algebra. {\bf 212}, 1999, 495-514.


 \bibitem[M]{M} M. Miyamoto, $ C_2$-cofiniteness of cyclic-orbifold models, {\em  Comm. Math. Phys.} {\bf  335}  (2015), 1279-1286.

\bibitem[MT]{MT}M. Miyamoto and K. Tanabe, Uniform product of $A_{g,n}(V)$ for an orbifold
model V and $G$-twisted Zhu algebra, {\em J. Algebra} {\bf 274} (2004), 80-96.
\bibitem[NSW]{NSW} S.-H. Ng, A. Schopieray, Y. Wang, Higher Gauss sums of modular categories,
Selecta Math. (N.S.) 25 (2019), no. 4, Art 53, 32 pp.


\bibitem[S]{S}J. Serre, A Course in Arithmetic, Graduate Texts
in Mathematics\emph{,} {\bf 7}, Springer-Verlag 1973.

\bibitem[T]{T} K. Tanabe, On intertwining operators and finite automorphism
groups of vertex operator algebras, {\em J. Alg.} {\bf 287} (2005), 174-198.

\bibitem[V]{V}
E. Verlinde, Fusion rules and modular transformation in 2D conformal field theory, {\em Nucl. Phys.}
{\bf B300} (1988), 360-376.
\bibitem[W]{W}
X.-G. Wen, Topological orders in rigid states, { \em Internat. J. Modern Phys. B} {\bf 4} (1990), no.2,  (1990), 239--271.
\bibitem[X]{X} F. Xu, Algebraic orbifold conformal field theories, {\em Proc. Natl. Acad. Sci.} USA {\bf 97} (2000), 14069-14073.

\bibitem[Xu]{Xu}
X. Xu, Introduction to Vertex Operator Superalgebras and Their Modules,
{\em Mathematics and its Applications,} {\bf Vol. 456},
 Kluwer Academic Publishers, Dordrecht, 1998.



\bibitem[Z]{Z}
Y. Zhu, Modular invariance of characters of vertex operator algebras,
{\em J. Amer, Math. Soc.}  {\bf 9} (1996), 237-302.
\end{thebibliography}
\end{document}